\documentclass[11pt]{amsart}
\usepackage{geometry}                
\usepackage[latin1]{inputenc}   
\usepackage{mathpazo,cite}  
\usepackage[dvipsnames,cmyk]{xcolor}
\usepackage{url}
\usepackage{amsmath,amssymb,latexsym,mathrsfs,amssymb,amsthm,amsfonts}
\usepackage{enumerate}
\usepackage{enumitem,overpic}
\usepackage{mathtools,mathdots} 
\usepackage[all]{xy}
\usepackage{tikz, tkz-euclide}
\usepackage{setspace}
\usepackage{algorithm}
\usepackage{algpseudocode}

\renewcommand{\vec}{\mathbf}

\newcommand{\lvertiii}{\left\vert\kern-0.25ex\left\vert\kern-0.25ex\left\vert}
\newcommand{\rvertiii}{\right\vert\kern-0.25ex\right\vert\kern-0.25ex\right\vert}

\usepackage[colorlinks=true, linkcolor=MidnightBlue, citecolor=MidnightBlue, urlcolor=MidnightBlue]{hyperref}

\newtheorem{theorem}{Theorem}[section]
\newtheorem{lemma}[theorem]{Lemma}

\newtheorem{corollary}[theorem]{Corollary}

\theoremstyle{definition}
\newtheorem{definition}[theorem]{Definition}

\theoremstyle{remark}
\newtheorem{remark}[theorem]{Remark}

\DeclareMathOperator{\supp}{supp}

\DeclareMathOperator{\tr}{tr}

\newcommand{\bmu}{{\boldsymbol{\mu}}}
\newcommand{\bnu}{{\boldsymbol{\nu}}}

\newcommand{\btau}{{\boldsymbol{\tau}}}
\newcommand{\bgamma}{{\boldsymbol{\gamma}}}

\newcommand{\RightComment}[1]{\hfill {\(\triangleright\) #1}}

\title[Banded Hermitian Matrices, Matrix Orthogonal Polynomials, and the Toda Lattice]{Banded Hermitian Matrices, Matrix Orthogonal Polynomials, and the Toda Lattice}

\author{Charbel Abi Younes}
\address{Charbel Abi Younes: Department of Applied Mathematics, University of Washington, Seattle, WA, USA}
\email{cyounes@uw.edu}

\author{Thomas Trogdon}
\address{Thomas Trogdon: Department of Applied Mathematics, University of Washington, Seattle, WA, USA}
\email{trogdon@uw.edu}

\begin{document}
\begin{abstract}
    We study the direct and inverse spectral theory for a class of finite Hermitian banded matrices. Using the theory of matrix orthogonal polynomials, we provide an explicit procedure for reconstructing a banded matrix from a matrix-valued measure that encodes its spectral data. We establish necessary and sufficient conditions for a measure to be the spectral measure of a matrix in the examined class. We further analyze the connections between this spectral analysis, block tridiagonalization algorithms, and the Toda lattice evolution on banded matrices.
\end{abstract}

\maketitle

\section{Introduction}

Spectral analysis of Jacobi matrices plays a fundamental role in many areas of mathematics. In mathematical physics, inverse spectral theory is used to study integrable nonlinear systems such as the Toda lattice, where the equations of motion are expressed as an isospectral deformation of a Jacobi matrix. In numerical linear algebra, 
the bijection between Jacobi matrices and spectral measures provides a rigorous framework for analyzing Krylov methods, such as the Lanczos algorithm\cite{chen_stability_2024,younes_lanczos-based_2025} and conjugate gradient method\cite{ding_conjugate_2022,deift_conjugate_2021,paquette_universality_2023,ding_riemannhilbert_2024}, using the theory of orthogonal polynomials. In this paper, we aim to examine the spectral and inverse spectral analysis of a broader class of Hermitian banded matrices. We show that working with general bandwidths gives rise to spectral measures that are matrix-valued, connecting naturally to the theory of matrix orthogonal polynomials. Our objective is to fully characterize the connection between these objects by extending the classical strategies used in the tridiagonal case. 

Explicitly, we study the spectral theory of finite $N\times N$ matrices of the form
\begin{align}\label{eq:BandedMatDef}
\vec J=\begin{bmatrix}
\vec A_0 & \vec B_0^* & & & & \\
\vec B_0 & \vec A_1 & \vec B_1^* & & & \\
& \vec B_1 & \vec A_2 & \vec B_2^* & & \\
& & \vec B_2 & \vec A_3 & \ddots & \\
& & & \ddots & \ddots & \vec B_{n-2}^* \\
& & & & \vec B_{n-2} & \vec A_{n-1}
\end{bmatrix},
\end{align}
where, for each $j$, $\vec A_j=\vec A_j(\vec J)$ is Hermitian and each $\vec B_j(\vec J)$ has full rank. We assume that 
\begin{equation}\label{eq:BlocksDim}
\begin{aligned}
    \vec A_j(\vec J) \in \mathbb C^{k\times k} ~ \text{for} ~ j=0,1,\dots,n-2, \quad \vec A_{n-1}(\vec J) \in \mathbb C^{(k-\ell)\times(k-\ell)},\\
    \vec B_j(\vec J)\in \mathbb C^{k\times k} ~\text{for}~ j=0,1,\dots,n-3, \quad \vec B_{n-2}(\vec J)\in \mathbb C^{(k-\ell) \times k},
\end{aligned}
\end{equation}
so that $N = nk - \ell$ with $0\leq\ell< k$. We further assume that $\vec B_j(\vec J)$ is in row echelon form with positive pivots. The precise class of matrices satisfying these conditions is formalized in Definition~\ref{def:JkN}. 

We define a spectral map that assigns to each $\vec J$ a $k\times k$ matrix-valued measure constructed from the eigenvalues of $\vec J$ and the first $k$ components of its normalized eigenvectors. Our main results provide a complete characterization of this map, showing that the associated (matrix) inner product is nondegenerate for polynomials of degree at most $n-2$ and identifying the exact rank deficiency for those of degree $n-1$ (Theorems~\ref{thm:Nd} and \ref{thm:NdEquiv}). We further show that this spectral map is injective (Theorem~\ref{thm:Injective}) and give a reconstruction procedure for the inverse spectral problem using the theory of matrix orthogonal polynomials, proving that each banded matrix in this class is uniquely determined by its spectral measure (Theorem~\ref{thm:InvSpec}). This inverse spectral analysis also allows us to characterize the range of the spectral map and to formulate necessary and sufficient conditions (Corollary~\ref{cor:NecSuffCond}) for a measure to correspond to a matrix of the form \eqref{eq:BandedMatDef}.

Although techniques and results for related classes of matrices have appeared in the literature (see Section~\ref{sec:lit}), to our knowledge, no prior work has leveraged the connection between banded matrices and matrix orthogonal polynomials for finite matrices, particularly in cases where the final blocks are smaller than the preceding ones. An important objective of this work is to also explore the implications of these results in other applications, including the equivalence of block tridiagonalization algorithms and the evolution of the Toda lattice on banded matrices. These connections have not been thoroughly investigated and provide a natural generalization of the classical connections between Jacobi matrices, the Toda flow, and Krylov methods.

The remainder of this paper is organized as follows. In the rest of this section, we review the relevant literature, discuss the connection between banded matrices and block tridiagonalization algorithms such as block Lanczos and the Householder algorithms, and provide background on the Toda lattice, briefly highlighting its generalization to banded matrices. In Section~\ref{sec:measures&polynomials}, we introduce the notion of matrix-valued measures, examine their properties, and define matrix orthogonal polynomials along with their recurrence relations. Section~\ref{sec:spectralmap} is devoted to the spectral analysis of matrices of the form \eqref{eq:BandedMatDef}. We define the spectral map, characterize its range, and establish its injectivity. We also develop the inverse spectral theory and present an approach to recovering banded matrices from their spectral measure. Finally, Section~\ref{sec:toda} provides a detailed discussion of the Toda flow on banded matrices and analyzes the evolution of the spectral measure.

\subsection{Related work}\label{sec:lit}

While the spectral analysis of Jacobi matrices \cite{deift_orthogonal_2000,gesztesy_m-functions_1997,berezanskii_expansions_1968,beckert_f_1966,antony_inverse_1994,gesztesy_isospectral_1996,hochstadt_construction_1974,hochstadt_construction_1979,deift_determination_1984,masson_spectral_1991,ferguson_construction_1980,de_boor_numerically_1978,teschl_trace_1998,akhiezer_classical_2020} has been studied in great detail, techniques for analyzing the spectral theory of general banded matrices\cite{zagorodnyuk_direct_2008,Kud98,Kud99,berezanskii_expansions_1968,beckert_f_1966,marchenko_inverse_2018,kudryavtsev_inverse_2017,kudryavtsev_inverse_2017-1,branquinho_spectral_2023,damanik_analytic_2014,biegler-konig_construction_1981,beckermann_spectral_2003,mattis_construction_1981} are less developed, with most progress appearing only in recent years. In \cite{duran_generalization_1993,dette_matrix_2002,beckert_f_1966,berezanskii_expansions_1968}, matrix orthogonal polynomials are used to establish a connection between matrix-valued measures and finite Hermitian block tridiagonal matrices with equal sized blocks, extending the classical spectral theory for Jacobi matrices. This approach was further developed in~\cite{damanik_analytic_2014}, where infinite matrices of this structure are analyzed thoroughly. To the best of our knowledge, this idea has not been applied to finite banded Hermitian matrices of the form \eqref{eq:BandedMatDef}, where the last block may have a degenerate size. This generalization presents additional challenges, arising from the associated matrix-valued measure inducing only a quasi-inner product and the corresponding orthogonal polynomials potentially having degenerate norms. Adressing these difficulties forms the focus of the present work.


Alternative techniques have been developed to analyze the direct and inverse spectral theory of broader classes of banded matrices. In~\cite{kudryavtsev_inverse_2017,kudryavtsev_inverse_2017-1}, the authors study finite and infinite real symmetric banded matrices whose off-diagonal entries vanish beyond a certain index.
The approach in~\cite{kudryavtsev_inverse_2017,kudryavtsev_inverse_2017-1} is based on the linear interpolation theory for vector polynomials introduced in~\cite{kudryavtsev_linear_2015}. This interpolation theory enables the extension of the results in~\cite{Kud98,Kud99} from pentadiagonal matrices to matrices of arbitrary bandwidth and leads to a unique reconstruction procedure for the inverse spectral problem, an approach that differs from the orthogonal polynomial methodology used in the present work. 

A different approach is presented in~\cite{marchenko_inverse_2018}, where the spectral theory of finite Hermitian block matrices is developed. This book introduces the notion of \emph{restricted spectral data on a completely extendable set}, which refers to a subset of the spectral data that is sufficient to uniquely reconstruct the corresponding Hermitian matrix. Theorem 8 in \cite[Section 15]{marchenko_inverse_2018} gives the necessary and sufficient conditions for such a set to exist and provides an reconstruction procedure. It is worth noting, however, that the results in~\cite{marchenko_inverse_2018} assume all blocks are of the same size.

The analysis in~\cite{branquinho_spectral_2023} addresses bounded banded operators that admit a positive bidiagonal factorization after an appropriate shift. This approach generalizes the spectral theorem beyond the setting of self-adjoint or normal operators and builds on ideas from the theory of oscillatory matrices. The methods in this work are closely related to multiple orthogonal polynomials.

In this paper, we extend the matrix orthogonal polynomial approach to finite Hermitian banded matrices whose structure is given by \eqref{eq:BandedMatDef}. This class of matrices is covered by the approach in~\cite{marchenko_inverse_2018} only when all the diagonal and off-diagonal blocks have the same size, so our method provides a more general framework for cases where $\ell>0$ in \eqref{eq:BlocksDim}. Additionally, our work gives a simplified and basic solution to the inverse spectral problem, drawing inspiration from the spectral analysis of Jacobi matrices and avoiding the complexities of linear interpolation theory in~\cite{kudryavtsev_inverse_2017,kudryavtsev_inverse_2017-1} and the theory of multiple orthogonal polynomials in~\cite{branquinho_spectral_2023}.

\subsection{Equivalence of block tridiagonalization algorithms}

The block Lanczos algorithm \cite{rice_block_1977,underwood_iterative_1975} is an iterative procedure for constructing a block tridiagonal approximation of a Hermitian matrix. In its simplest form, it is given by Algorithm~\ref{a:BlockLanczos} in Appendix~\ref{appendix:Algs}. Given a Hermitian matrix $\vec A\in \mathbb C^{N\times N}$ and an initial block $\vec V\in \mathbb C^{N\times k}$, the block Lanczos iteration at step $n\leq \lceil N/k\rceil$ produces a block tridiagonal matrix $\vec J_n$ with the structure of \eqref{eq:BandedMatDef} and a sequence of matrices $\vec V_1, \dots, \vec V_n \in \mathbb{C}^{N \times k}$ such that\footnote{For integers $m,n$, the matrix $\vec I_{m\times n}$ denotes the $m\times n$ identity matrix, and we abbreviate $\vec I_n := \vec I_{n\times n}$.} $\vec V_i^*\vec V_j = \delta_{ij} \vec I_k$ and
\begin{align*}
    \vec J_n = \vec Q_n^*\vec A \vec Q_n, \qquad \vec Q_n = \begin{bmatrix}
        \vec V_1 & \cdots & \vec V_n
    \end{bmatrix}.
\end{align*}
The columns of $\vec Q_n$ form an orthonormal basis for the degree $n$ block Krylov subspace defined as
\begin{align*}
    \mathcal{K}_{n+1}(\vec A,\vec V) := \mathrm{span}\left\{\vec V, \vec A\vec V, \dots, \vec A^n \vec V \right\}.
\end{align*}
Here, the span is interpreted as the span of all columns of the matrices $\vec V,\vec A\vec V,\dots,\vec A^n \vec V$. It is important to note that if 
\begin{align*}
    \vec K(\vec A) := \begin{bmatrix}
        \vec I_{N\times k},& \vec A \vec I_{N\times k},& \dots,& \vec A^{\lceil N/k\rceil}\vec I_{N\times k}
    \end{bmatrix}
\end{align*}
has full rank, then block Lanczos on $\vec A$ with starting block $\vec I_{N\times k}$ runs for $\lceil N/k \rceil$ iterations and is said to \textit{run to completion}. On the other hand, if $n_0$ is the first index for which $\mathcal{K}_{n_0}(\vec A,\vec V)=\mathcal{K}_{n_0+1}(\vec A,\vec V)$, then the algorithm \textit{terminates early} at step $n_0$, as further iterations no longer generate linearly independent vectors.

Block tridiagonalization can also be achieved by modifying the classical Householder tridiagonalization algorithm using a different elimination pattern. For a vector $\vec v=[v_1^*,\dots,v_n^*]^*\in\mathbb C^n$, the associated Householder reflector is defined by
\begin{align*}
    \vec H(\vec v)
    = \vec I_n - 2 \vec w \vec w^*,
    \qquad 
    \vec w = \frac{\vec u}{\|\vec u\|},
    \qquad 
    \vec u = \frac{|v_1|}{v_1}\|\vec v\|\vec e_1 + \vec v,
\end{align*}
with the convention that $|0|/0=1$. The Householder procedure uses these reflectors to construct a sequence of unitary matrices $\{\vec Q_j\}_{j=1}^N$ that sequentially bring the leading $n\times n$ principal subblocks of 
\begin{align}\label{eq:HouseholderTransf}
    \vec Q_{n}\cdots \vec Q_1 \vec A \vec Q_1^* \cdots \vec Q_{n}^*
\end{align}
to the desired block tridiagonal form. At step $n$, the matrix $\vec Q_n$ is constructed as
\begin{align*}
    \vec Q_n
    =
    \begin{bmatrix}
        \vec I_{k+n-1} & \vec 0 \\
        \vec 0 & \vec H_n
    \end{bmatrix}, \qquad \vec H_n= \vec H(\vec v_n),
\end{align*}
where $\vec v_n$ is the vector formed by the last $N-k-n+1$ entries of the $n$-th column of \eqref{eq:HouseholderTransf}. When $\vec Q_n^*$ is applied on the left of \eqref{eq:HouseholderTransf}, it introduces zeros into rows $k+n+1,\dots,N$ of the $n$-th column. Multiplication on the right by $\vec Q_n$ preserves these zeros, and by symmetry the corresponding entries in the $n$-th row are also zero, yielding the desired block tridiagonal form. 

A brief review of the Householder reduction to block tridiagonal form is provided in Appendix~\ref{appendix:Algs}. The algorithm applies Householder reflectors sequentially, one transformation at a time. For improved computational efficiency, especially on modern architectures, several consecutive reflectors can be accumulated and applied simultaneously using block representations \cite{golub_matrix_2013,bischof_orthogonal_1996,dongarra_block_1989}. Such implementations rely on compact representations of products of
Householder reflectors, such as the $\vec W\vec Y^*$ and compact $\vec W\vec Y^*$ representations\cite{bischof_wy_1987,schreiber_storage-efficient_1989}. A related but conceptually distinct approach is based on \emph{block Householder reflectors}, defined as matrices of the form
\begin{align*}
    \vec H(\vec W) = \vec I_n - 2 \vec W \vec W^* ,
\end{align*}
where $\vec W\in \mathbb C^{N\times k}$ satisfies $\vec W^* \vec W = \vec I_k$. These reflectors generalize classical Householder transformations, but they do not, in general, preserve the triangular structure of the off-diagonal blocks described in \eqref{eq:BandedMatDef}. For further details, see \cite{schreiber_block_1988}.

The following result establishes an equivalence between block Lanczos and the Householder reduction. Although this equivalence is classical \cite[Theorem 2.1]{bloemendal_limits_2016}, we present a simple alternative argument based on the spectral theory of banded Hermitian matrices.

\begin{theorem}
    If block Lanczos (Algorithm~\ref{a:BlockLanczos}) applied to a Hermitian matrix $\vec A\in \mathbb C^{N\times N}$ with starting block $\vec I_{N\times k}$ runs to completion, then it produces the same block tridiagonal matrix as the Householder procedure (Algorithm~\ref{a:BlockHouseholder}).
\end{theorem}
\newcommand{\JkN}{\mathcal J_{k,N}}
\begin{proof}
    Let $\vec J_1$ and $\vec J_2$ denote the block tridiagonal matrices produced by the block Lanczos algorithm and the Householder procedure, respectively, and write
    \begin{align}\label{eq:SimilarityEqsLanczosHouseholder}
        \vec J_1 = \vec Q_1^* \vec A \vec Q_1, \quad \vec J_2 = \vec Q_2^* \vec A \vec Q_2,
    \end{align}
    where $\vec Q_1,\vec Q_2$ are unitary matrices whose first $k$ columns are $\vec I_{N\times k}$. Consider the spectral map $\varphi$ defined in \eqref{eq:SpectMapDef1} below, and let $\bmu_1 = \varphi(\vec J_1)$ and $\bmu_2 = \varphi(\vec J_2)$. Since $\vec J_1$ and $\vec J_1$ have the same eigenvalues and identical first $k$ eigenvector entries, it follows that $\bmu_1 = \bmu_2$. Injectivity of $\varphi$ in Theorem~\ref{thm:Injective} therefore gives $\vec J_1=\vec J_2$.
\end{proof}

\subsection{Toda flow}\label{sec:todabackground}

The finite Toda lattice, introduced by Morikazu Toda in 1967 \cite{toda_vibration_1967,toda_theory_1989}, is a completly integrable model for a nonlinear one-dimensional crystal. The integrability of the Toda lattice was established independently by Flaschka\cite{flaschka_toda_1974} and Manakov\cite{manakov_complete_1975}, who proved that the system can be written as
\begin{align}\label{eq:toda}
    \partial_t \vec X = [\vec X,\vec B(\vec X)] = \vec X\vec B(\vec X)-\vec B(\vec X)\vec X, \quad \vec X(0)=\vec X_0,
\end{align}
where $\vec X_0 \in \mathbb C^{N\times N}$ is a Jacobi matrix. The matrix $\vec B(\vec X)$ is defined by
\begin{align}\label{eq:Bdef}
    \vec B(\vec X) = \vec X_- - \vec X_-^T,
\end{align}
with $\vec X_-$ denoting the strictly lower triangular part of $\vec X$.

Under the Toda flow, the eigenvalues of $\vec X$ remain constant in time, and the first entries of the eigenvectors evolve in a simple way. Let $\lambda_j(t)$ be the eigenvalues of $\vec X(t)$ and $\vec v_{1,j}(t)$ be the first entry of the $j$-th normalized eigenvector, and assume without loss of generality that \footnote{For a Jacobi matrix, all eigenvalues are simple, and the first component of each eigenvector is nonzero.} $\vec v_{1,j}(t)>0$. For each $j = 1, \dots, N$, we have
\begin{align}\label{eq:JacEvalEvecEvol}
\lambda_j(t) = \lambda_j, \quad \text{and} \quad
\vec v_{1,j}(t) = \frac{\vec v_{1,j}(0) e^{\lambda_j t}}{\left( \sum_{i=1}^N \vec v_{1,i}^2(0), e^{2\lambda_i t} \right)^{1/2}},
\end{align}
where $\lambda_1\geq \lambda_2\geq \dots\geq \lambda_N$ are the eigenvalues of $\vec X_0$. This leads to a simple expression for the spectral measure of $\vec X(t)$, given by
\begin{align}\label{eq:JacSpectralMeasure}
    \mu_{\vec X(t)} = \sum_{j=1}^N w_j(t)\delta_{\lambda_j}, \qquad w_j(t) = |\vec v_{1,j}(t)|^2 = \frac{e^{2\lambda_j t} w_j(0)}{\sum_{j=1}^N e^{2\lambda_i t} w_i(0)} \quad  \text{for }j=1,\dots,N.
\end{align}

The evolution of $\mu_{\vec X(t)}$ provides a direct procedure for solving the finite Toda lattice using inverse spectral methods. Starting from an initial Jacobi matrix $\vec X_0$, one computes its eigenvalues and the first components of its eigenvectors, evolves the weights using an explicit exponential factor, and reconstructs the Jacobi matrix from the spectral measure. This solution process is remarkable because, despite the Toda lattice being nonlinear, its integrable structure provides spectral variables in which the evolution is explicit and simple.


There are limited results giving analogous formulae for the Toda lattice with banded matrix initial data. Results for pentadiagonal initial data appear in \cite{Nanda1982}, where the evolution of the first two eigenvector components is derived. However, to the best of our knowledge, an explicit characterization of Toda lattice in terms of the evolution of a matrix-valued measure has not been obtained for general bandwidth. Using the spectral theory developed in Sections~\ref{sec:spectralmap}, we show that the Toda flow preserves the banded structure of matrices of the form \eqref{eq:BandedMatDef} and induces a simple evolution on the associated matrix-valued spectral measure. Specifically, the spectral measure evolves as
\begin{align*}
    \bmu_{\vec X(t)} = \sum_{j=1}^N \vec L^{-1}(t) \big( e^{2\lambda_j t} \vec v_j(0) \vec v_j(0)^* \big) \vec L^{-*}(t) \, \delta_{\lambda_j},
\end{align*}
where $\{\lambda_j\}_{j=1}^N$ are the eigenvalues of $\vec X_0$,  $\vec v_j(0)$ denotes the first $k$ components of the corresponding eigenvector, and $k$ is the bandwidth or block size of $\vec X_0$. The lower triangular matrix $\vec L(t)$ has positive diagonal entries and satisfies
\begin{align*}
    \sum_{j=1}^N e^{2\lambda_j t} \vec v_j(0)\vec v_j(0)^* = \vec L(t) \vec L(t)^*,
\end{align*}
and thus serves to give the Cholesky factorization of the sum. This shows that the Toda flow is also fully determined by the spectral data for banded matrices. This result is presented in Section~\ref{sec:toda}.

\section{Matrix-valued measures and polynomials}\label{sec:measures&polynomials}

A (positive) matrix-valued measure $\bmu$ on $\mathbb R$ is a countably-additive set function
\begin{align*}
    \bmu :\mathcal B(\mathbb R) \to \mathbb C^{k \times k}_+,
\end{align*}
where $\mathcal B(\mathbb R)$ denotes the Borel $\sigma$-algebra on $\mathbb R$ and $\mathbb C^{k \times k}_+$ is the set of Hermitian positive semi-definite $k\times k$ matrices.  The measure $\bmu$ is said to be normalizable if $\|\bmu(\mathbb R)\| < \infty$ and $\det \bmu(\mathbb R) \neq 0$.  In this case, we assume\footnote{The normalized measure $\bnu$ is given by $\bnu(x) = \bmu(\mathbb{R})^{-1/2} \bmu(x) \bmu(\mathbb{R})^{-1/2}$. Equivalently, if $\bmu(\mathbb{R}) = \vec L \vec L^*$ is the Cholesky decomposition of the total mass, then $\bnu$ can be expressed as $\bnu(x) = \vec L^{-1} \bmu(x) \vec L^{-*}$.}
\begin{align*}
    \bmu(\mathbb R) = \vec I_k,
\end{align*}
and we write $\bmu \in \mathcal P_k(\mathbb R)$. A point $x$ is said to be in the support of $\bmu$ if $\bmu(B) \neq \vec 0$ for every open set $B$ that contains $x$. The set of all such points is denoted by $ \supp \bmu$. 

A matrix-valued measure $\bmu$ induces both left and right quasi-inner products on matrix-valued Borel-measurable functions $\vec F,\vec G\colon\mathbb R\to\mathbb C^{k\times k}$, defined, respectively, by
\begin{align}
    \langle \vec F,\vec G\rangle_L := \int_{\mathbb R}\vec F(x)\bmu(\mathrm d x)\vec G^*(x), \quad \langle \vec F,\vec G\rangle_R := \int_{\mathbb R}\vec F^*(x)\bmu(\mathrm d x)\vec G(x).
\end{align}
These products are related via $\langle \vec F, \vec G \rangle_L = \langle \vec F^*, \vec G^* \rangle_R$. Throughout the paper, we work mainly with the right quasi-inner product and, with slight abuse of notation, write $\langle \vec F, \vec G \rangle_\bmu$ in place of $\langle \vec F, \vec G \rangle_R$. The right product satisfies the following properties:
\begin{enumerate}
    \item $\langle \vec F, \vec G \rangle_\bmu = \langle \vec G, \vec F \rangle_\bmu^*$,
    \item For $\vec C\in \mathbb C^{k\times k}$, $\langle \vec F,\vec G \vec C\rangle_\bmu = \langle \vec F,\vec G \rangle_\bmu \vec C$ and $\langle\vec F \vec C,\vec G\rangle_\bmu = \vec C^* \langle \vec F,\vec G \rangle_\bmu$,
    \item $\langle \vec F, \vec F \rangle_\bmu \in  \mathbb C^{k\times k}_+$,
    \item $\| \vec F \| :=  (\tr \langle \vec F, \vec F \rangle_\bmu)^{1/2}$ defines a semi-norm.
\end{enumerate}
We note that $\|\vec F\|$ cannot be expected to be a norm without additional assumptions. For example, consider
\begin{align*}
    \bmu = \sum_{i = 1}^n \vec w_i \vec w_i^* \delta_{x_i},
\end{align*}
for $x_i < x_{i+1}$ and $\sum_i \vec w_i \vec w_i^* = \vec I_k$. Then for any function $\vec F$ satisfying $\vec F(x_i) \vec w_i = \vec 0$, we have $\langle \vec F, \vec F \rangle_\bmu  = \vec 0$. This motivates the notion of $n$-definiteness, which guarantees that a sequence of $n$ orthonormal polynomials is well-defined (see Section \ref{subsec:OrthonormalPolys}).
\begin{definition}
    A measure $\bmu \in \mathcal P_k(\mathbb R)$ is called \emph{$n$-definite} if $\|\vec P\| \neq 0$ for every matrix-valued polynomial $\vec P$ of degree at most $n-2$ satisfying $\vec P(x) \neq \vec 0$ for some $x \in \mathbb R$.
\end{definition}


\subsection{Monic Orthogonal Polynomials}

Consider a sequence of monic matrix-valued polynomials $\vec \Pi_0(x),\vec \Pi_1(x),\vec \Pi_2(x),\dots$ defined by
\begin{align}\label{eq:MonicDef}
    \vec \Pi_j(x) = \vec \Pi_j(x,\bmu) = \vec I_k~x^j + \btau_j(\bmu) x^{j-1} + \text{ lower order terms}, \quad \btau_j:=\btau_j(\bmu)\in \mathbb C^{k\times k}.
\end{align}
These polynomials are said to be orthogonal with respect to the quasi-inner product $\langle \cdot,\cdot \rangle_\bmu$ if\footnote{We use $\diamond$ as a placeholder for the variable of a function. For instance, $\frac{1}{\diamond}$ denotes the function $f(x) = \frac{1}{x}$.}
\begin{align}\label{eq:MonicOrthogCond}
    \langle \vec \Pi_i(\diamond,\bmu),\vec \Pi_j(\diamond,\bmu)\rangle_\bmu = \bgamma_j \delta_{ij}, \quad \text{where $\bgamma_j:= \bgamma_j(\bmu)=\langle \vec \Pi_j(\diamond,\bmu),\vec \Pi_j(\diamond,\bmu) \rangle_\bmu \in \mathbb C_+^{k\times k}$}.
\end{align}
We now show that, when $\bmu$ is $n$-definite, monic orthogonal polynomials of degree at most $n-1$ are uniquely determined, and therefore $\{\bgamma_j(\bmu)\}_{j=0}^{n-1}$ and $\{\tau_j(\bmu)\}_{j=0}^{n-1}$ are well-defined. To establish this, we first note that any polynomial can be expressed as a linear combination of monic polynomials.
\begin{lemma}\label{lem:low_deg_expansion}
    Let $\{\vec \Pi_j\}_{j\geq0}$ be a sequence of monic matrix-valued polynomials. For every polynomial $\vec P\in \mathbb C^{k\times k}$ of degree $p$, there exists unique matrices $\{\Phi_j\}_{j=0}^{p}$ with $\Phi_j \in \mathbb{C}^{k\times k}$ such that $\vec P(x) =\sum_{j=0}^{p}\vec \Pi_j(x) \Phi_j$.
\end{lemma}
\begin{theorem}\label{Thm:DefiniteProp}
    If the $\bmu$ is $n$-definite, then the monic orthogonal polynomials $\{\vec \Pi_j(\diamond,\bmu)\}_{j=0}^{n-1}$ are uniquely defined and $\det \bgamma_j(\bmu) \neq 0$ for $j \leq n-2$.
\end{theorem}
\begin{proof}
    Existence of the monic orthogonal polynomials $\vec \Pi_0,\dots, \vec \Pi_{n-1}$ follows recursively by applying the Gram-Schmidt process to the sequence ($x\mapsto\vec I_k$, $x\mapsto x\vec I_k$, $\dots$,  $x\mapsto x^{n-1}\vec I_k$), provided that each $\bgamma_j$ is nonsingular for $1\leq j<n-1$. 
    
    To show $\mathrm{det}~\bgamma_{j}\neq 0$ for $j<n-1$, suppose for contradiction that $j$ is the first index where $\bgamma_j$ is singular. In this case, there exists a nonzero vector $\vec v$ such that $\bgamma_j \vec v = \vec 0$. Define $\vec \Pi = \vec \Pi_j\vec v\vec v^*$, then $\vec \Pi\neq \vec 0$ since its leading coefficient is $\vec v\vec v^*\neq \vec 0$. However, we have $\langle \vec \Pi,\vec \Pi \rangle_\bmu = \vec v\vec v^*\bgamma_j \vec v\vec v^* = \vec 0$, which contradicts the fact that $\bmu$ is $n$-definite, and thus $\det \bgamma_j \neq 0$.

    To establish uniqueness, let $\tilde{\vec\Pi}_j$ be another monic orthogonal polynomial of degree $j$. Note that $\vec \Pi_j-\tilde{\vec \Pi}_j$ is a polynomial of degree of $j-1$, and hence
    \begin{align*}
        \langle \vec \Pi_j-\tilde{\vec \Pi}_j, \vec \Pi_j-\tilde{\vec \Pi}_j\rangle_\bmu = \langle \vec \Pi_j , \vec \Pi_j-\tilde{\vec \Pi}_j \rangle_\bmu - \langle \tilde{\vec \Pi}_j , \vec \Pi_j-\tilde{\vec \Pi}_j \rangle_\bmu = \vec 0,
    \end{align*}
    where the second equality follows from Lemma~\ref{lem:low_deg_expansion}.
    Finally, since $\bmu$ is $n$-definite, it follows that $\tilde{\vec \Pi}_j = \vec \Pi_j$.
\end{proof}

The monic orthogonal polynomials satisfy a number of well-known properties (see, for instance, \cite{damanik_analytic_2014,sinap_orthogonal_1996,krein_fundamental_1971,miranian_matrix-valued_2005,duran_orthogonal_1995,sinap_polynomial_1994}). We highlight below a fundamental result that will be important for our work.
\begin{lemma}
    If $\bmu$ is $n$-definite and $\vec P$ has degree $p\leq n-2$, then the coefficients $\{\Phi_j\}_{j=0}^p$ in Lemma~\ref{lem:low_deg_expansion} are given by $\Phi_j = \bgamma_j^{-1}(\bmu) \langle \vec \Pi_j,\vec P \rangle_\bmu$.
\end{lemma}
\begin{proof}
    The statement follows from Lemma~\ref{lem:low_deg_expansion}, the orthogonality relations in \eqref{eq:MonicOrthogCond}, and Theorem~\ref{Thm:DefiniteProp}, which ensures that $\bgamma_j$ is invertible for $j\leq n-2$.
\end{proof}

\begin{theorem}\label{thm:MonicRec}
    Assume that $\bmu$ is $n$-definite. Then the associated monic orthogonal matrix polynomials obey the following three-term recurrence relation,
    \begin{align}\label{eq:MonicRec}
        x\vec \Pi_j(x) = \vec \Pi_{j+1}(x) + \vec \Pi_j(x) \vec C_j + \vec \Pi_{j-1}(x) \vec D_j , \quad \text{for } j=0,1,\dots,n-2
    \end{align}
    where $\vec \Pi_{-1}\equiv 0$, $\bgamma_{-1} = \vec I_k$,
    \begin{align*}
        \vec C_j(\bmu) = \bgamma_{j}^{-1}\langle \diamond\vec \Pi_j(\diamond),\vec \Pi_j \rangle_\bmu = \btau_j(\bmu)-\btau_{j+1}(\bmu), \quad \vec D_j(\bmu) = \bgamma_{j-1}^{-1}(\bmu)\langle \diamond\vec \Pi_{j-1}(\diamond),\vec \Pi_j \rangle_\bmu = \bgamma_{j-1}^{-1}(\bmu)\bgamma_j(\bmu).
    \end{align*}
\end{theorem}
\begin{proof}
    Since $x\vec \Pi_j(x)-\vec \Pi_{j+1}(x)$ is a polynomial of degree $j$ with leading coefficient $\btau_j-\btau_{j-1}$, Lemma~\ref{lem:low_deg_expansion} implies that
    \begin{align*}
        x\vec \Pi_j(x)-\vec \Pi_{j+1}(x) = \sum_{i=0}^j \vec \Pi_i(x) \Phi_i, \quad \text{with} \quad  \Phi_j = \btau_j-\btau_{j+1}.
    \end{align*}
    Taking the inner product of both sides with $\vec \Pi_\ell$ gives
    \begin{align*}
        \langle \vec \Pi_\ell, \diamond\vec \Pi_j(\diamond)-\vec \Pi_{j+1} \rangle_\bmu = \sum_{i=0}^j \langle \vec \Pi_\ell , \vec \Pi_i\rangle_\bmu \Phi_i, = \bgamma_\ell \Phi_\ell.
    \end{align*}
    However, the left-hand side can be rewritten as
    \begin{align*}
        \langle \vec \Pi_\ell, \diamond\vec \Pi_j(\diamond)-\vec \Pi_{j+1} \rangle_\bmu = \langle \diamond\vec \Pi_\ell(\diamond), \vec \Pi_j \rangle_\bmu - \langle \vec \Pi_\ell,\vec \Pi_{j+1} \rangle_\bmu.
    \end{align*}
    By the orthogonality condition \eqref{eq:MonicOrthogCond}, both terms vanish when $\ell\leq j-2$, and hence $\Phi_\ell = \vec 0$ in this case. Finally, we have
    \begin{align*}
        \Phi_{j-1} = \bgamma_{j-1}^{-1} \langle \diamond\vec \Pi_{j-1}(\diamond), \vec \Pi_j \rangle_\bmu = \bgamma_{j-1}^{-1}\left( \langle \vec \Pi_{j}, \vec \Pi_j \rangle_\bmu - \langle \diamond\vec \Pi_{j-1}(\diamond) - \vec \Pi_j, \vec \Pi_j \rangle_\bmu \right) = \bgamma_{j-1}^{-1}\bgamma_j,
    \end{align*}
    which completes our proof.
\end{proof}

\subsection{Orthonormal Polynomials}\label{subsec:OrthonormalPolys}

If $\bmu$ is $n$-definite, then a family of orthonormal matrix polynomials $\{\vec P_j(\diamond,\bmu)\}_{j=0}^{n-2}$ is given by
\begin{align}\label{eq:OpsNormalization}
\vec P_j(x) := \vec P_j(x,\bmu) =  \vec \Pi_j(x,\bmu) \vec \bgamma_j^{-1/2}(\bmu) \vec Q_j,
\end{align}
where $\vec \Pi_j(\diamond,\bmu)$ and $\bgamma_j(\bmu)$ are defined in \eqref{eq:MonicOrthogCond}, and $\vec Q_j$ is an arbitrary unitary matrix. It is important to note that, for every choice of $\{\vec Q_j\}_{j=0}^{n-2}$, we have $\langle\vec P_i,\vec P_j\rangle_\bmu = \vec I_k\delta_{ij}$, so orthonormal polynomials are determined only up to a right multiplication by a unitary matrix. Throughout our work, we fix $\vec Q_0 = \vec I_k$ so that $\vec P_0 = \vec I_k$.

Theorem~\ref{thm:MonicRec} implies that these polynomials follow a Hermitian recurrence relation
\begin{align}\label{eq:OrthonormRec}
    x \vec P_j(x) = \vec P_{j+1}(x) \vec B_{j} + \vec P_j(x)\vec A_j + \vec P_{j-1}(x)\vec B^*_{j-1}, \quad j=0,\dots,n-3, \quad \vec A_j = \vec A_j^*, \quad \det \vec B_j\neq 0,
\end{align}
with the convention $\vec P_{-1} \equiv \vec 0$ and $\vec B_{-1} \equiv \vec I_k$. The recurrence coefficients $\vec A_j = \vec A_j(\bmu)$ and $\vec B_j = \vec B_j(\bmu)$ are explicitly given by
\begin{align}\label{eq:RecCoefForm}
\vec A_j = \langle \vec P_j,\diamond \vec P_j(\diamond) \rangle_\bmu = \vec Q_j^*\bgamma_j^{1/2}\vec C_j\bgamma_j^{-1/2}\vec Q_j, \quad \text{and} \quad \vec B_j = \langle \vec P_{j+1}, \diamond \vec P_j(\diamond) \rangle_\bmu = \vec Q_{j+1}^*\bgamma_{j+1}^{-1/2}\vec D_{j+1}^* \bgamma_j^{1/2}\vec Q_j,
\end{align}
where $\vec C_j$ and $\vec D_j$ are defined in Theorem~\ref{thm:MonicRec}. Lemma~\ref{lem:OpsUniqNormalization} shows that there exists a unique choice of $\{\vec Q_j\}_{j=0}^{n-2}$ that produces a three-term recurrence in which $\vec B_j$, for $j=0,\dots,n-3$, is upper triangular with positive diagonal entries.

\begin{lemma}\label{lem:OpsUniqNormalization}
    Let $\bmu$ be $n$-definite, then there exists a unique family of matrix-valued orthogonal polynomials $\{\vec P_j(\diamond,\bmu)\}_{j=0}^{n-2}$, i.e. a unique choice of unitary matrices $\{\vec Q_j\}_{j=1}^{n-2}$ in \eqref{eq:OpsNormalization}, such that the matrices $\{\vec A_j(\bmu)\}_{j=0}^{n-3}$ are Hermitian and the matrices $\{\vec B_j(\bmu)\}_{j=0}^{n-3}$ are upper triangular with positive diagonal entries.
\end{lemma}
\begin{proof}
    Let $\{\vec{\tilde P}_j\}_{j=0}^{n-2}$ be any family of matrix orthonormal polynomials associated with $\bmu$, with recurrence coefficients $\{\vec{\tilde A}_j\}_{j=0}^{n-3}$ and $\{\vec{\tilde B}_j\}_{j=0}^{n-3}$. We first construct a family $\{\vec P_j\}$ inductively using unitary transformations to enforce the required conditions on the recurrence coefficients. Consider 
    \begin{align*}
        x\vec{\tilde P}_0(x) = \vec{\tilde P}_1(x) \vec{\tilde B}_0 + \vec{\tilde P}_0(x) \vec{\tilde A}_0,
    \end{align*}
    and let $\vec{\tilde B}_0 = \vec Q_1 \vec B_0$ be a QR decomposition, where $\vec Q_1$ is unitary and $\vec B_0$ is upper triangular with positive diagonal entries. Setting
    \begin{align*}
        \vec P_0 = \vec{\tilde P}_0, \quad \vec P_1 = \vec{\tilde P}_1 \vec Q_1, \quad \text{and} \quad \vec A_0 = \vec{\tilde A}_0,
    \end{align*}
    we find
    \begin{align*}
        x\vec P_0(x) = \vec P_1(x) \vec B_0 + \vec P_0(x) \vec A_0.
    \end{align*}
    Assume now that $\{\vec P_j\}_{j=0}^{i}$ have been constructed so that the required conditions on $\vec A_j$ and $\vec B_j$ hold for all $j<i$. Let $\vec Q_1,\dots,\vec Q_i$ be the unitary matrices used in the previous steps, and by construction these matrices satisfy
    \begin{align*}
        \vec{\tilde B}_j\vec Q_j = \vec Q_{j+1} \vec B_j, \quad \text{and} \quad \vec A_j = \vec Q_j^* \vec{\tilde A_j\vec Q_j} \quad \text{for $j=1,\dots,i-1$}.
    \end{align*}
    Consider the recurrence for $\vec{\tilde P}_{i+1}$, that is,
    \begin{align*}
        x\vec{\tilde P}_i(x) = \vec{\tilde P}_{i+1}(x) \vec{\tilde B}_i + \vec{\tilde P}_i(x) \vec{\tilde A}_i + \vec{\tilde P}_{i-1}(x) \vec{\tilde B}_{i-1}.
    \end{align*}
    Multiplying on the right by $\vec Q_i$ gives
    \begin{align*}
        x\vec P_i(x) = \vec{\tilde P}_{i+1}(x) \vec{\tilde B}_i \vec Q_{i} + \vec P_i(x) \vec Q_i^*\vec{\tilde A}_i \vec Q_i + \vec{\tilde P}_{i-1}\vec Q_{i-1}^*(x) \vec{\tilde B}_{i-1} \vec Q_i.
    \end{align*}
    Using a QR decomposition $\vec{\tilde B}_i \vec Q_i = \vec Q_{i+1}\vec B_i$, we define $\vec P_{i+1}=\vec{\tilde P}_{i+1}\vec Q_{i+1}$ and $\vec A_i = \vec Q_i^* \vec A_i \vec Q_i$. This guarantees that the $(i+1)$-th recurrence has the desired form, establishing the existence of a normalization. Uniqueness follows from the uniqueness of the QR decomposition for invertible matrices.
\end{proof}

For the remainder of this paper, we fix the orthonormal polynomials $\{\vec P_j(\diamond,\bmu)\}_{j=0}^{n-2}$ to be the unique sequence given by Lemma~\ref{lem:OpsUniqNormalization}. The $(n-1)$-th orthonormal polynomial\footnote{In this paper, we define the $(n-1)$-th orthonormal polynomial as the unique $k\times (k-\ell)$ polynomial $\vec P_{n-1}$, where $\mathrm{rank}~\gamma_{n-1}=k-\ell$, satisfying $\langle \vec P_i,\vec P_j \rangle_\bmu = \vec 0_{k\times(k-\ell)}$ for $i\neq j$ and $\langle \vec P_i,\vec P_i \rangle_\bmu = \vec I_{k-\ell}$ for $i=0,\dots,n-1$, with $\langle \vec P_{n-1},\diamond\vec P_{n-2}(\diamond) \rangle_\bmu$ is in row echelon form with positive pivots.} $\vec P_{n-1}(\diamond,\bmu)$ is not directly defined, since the normalization factor $\bgamma_{n-1}(\bmu)=\langle\vec\Pi_{n-1}(\diamond,\bmu),\vec \Pi_{n-1}(\diamond,\bmu) \rangle$ in \eqref{eq:OpsNormalization} is potentially singular. In the following, we construct $\vec P_{n-1}(\diamond,\bmu)$ when $\mathrm{rank}~\bgamma_{n-1}(\bmu) = k-\ell$ for some $0\leq \ell<k$. 

\begin{lemma}\label{lem:REF}
    Let $\vec A \in \mathbb C^{n\times n}$ be a positive semidefinite matrix of rank $k$, then there exists a unique upper triangular matrix $\vec R\in \mathbb C^{k\times n}$ in row echelon form and positive pivots such that $\vec A = \vec R^* \vec R$. 
\end{lemma}
\begin{proof}
    Since $\vec A$ is positive semidefinite of rank $k$, there exists $\vec X \in \mathbb C^{k\times n}$ such that $\vec A = \vec X^* \vec X$. Because $\vec X$ has full row rank, it can be factored as $\vec X = \vec Q \vec R$ where $\vec Q \in \mathbb C^{k\times k}$ is unitary and $\vec R \in \mathbb C^{k\times n}$ is in row echelon form with positive pivots. Substituting into $\vec A = \vec X^* \vec X$ gives $\vec A = \vec R^* \vec Q^* \vec Q \vec R =  \vec R^* \vec R$.

    Suppose $\vec{\tilde R}$ is another matrix in row echelon form with positive pivots such that $\vec A = \vec{\tilde R}^* \vec{\tilde R}$. The pivot locations of $\vec R$ and $\vec{\tilde R}$ must coincide; otherwise, the submatrix of $\vec R$ formed by the pivot columns is nonsingular while the corresponding submatrix of $\vec{\tilde R}$ is singular. Moreover, defining\footnote{Here $\vec A^\dagger$ denotes the Moore-Penrose pseudoinverse of $\vec A$.} $\vec U:=\vec R\tilde{\vec R}^\dagger\in\mathbb C^{k\times k}$, we have
    \begin{align*}
        \vec U^*\vec U = (\tilde{\vec R}^\dagger)^*\vec R^* \vec R\tilde{\vec R}^\dagger = (\tilde{\vec R} \tilde{\vec R}^\dagger)^* (\tilde{\vec R}\tilde{\vec R}^\dagger) = \vec I_k,
    \end{align*}
    and
    \begin{align*}
        \vec U \tilde{\vec R} = \vec R \tilde{\vec R}^\dagger \tilde{\vec R} = (\vec R \vec R^*)^{-1}(\vec R\vec R^*) \vec R \tilde{\vec R}^\dagger\tilde{\vec R} = (\vec R\vec R^*)^{-1}\vec R \tilde{\vec R}^* (\tilde{\vec R}\tilde{\vec R}^\dagger) \tilde{\vec R} = (\vec R \vec R^*)^{-1}(\vec R \vec R^*) \vec R = \vec R, 
    \end{align*}
    where we used the facts that $\vec R^*\vec R = \tilde{\vec R}^*\tilde{\vec R}$ and $\tilde{\vec R} \tilde{\vec R}^\dagger = \vec I_k$. We now claim that $\vec U$ is upper triangular. Indeed, if $\vec U_{ij}\neq 0$ for $j<i$, the $i$-th row of $\vec R$ is given by $\vec r_i = \sum_{j=1}^k \vec U_{ij}\vec{\tilde r}_j$, and thus has a nonzero entry before its pivot. Since $\vec U$ is unitary and all the pivots are positive, we conclude that $\vec U = \vec I_k$ and therefore $\vec R= \vec{\tilde R}$.
\end{proof}

\begin{theorem}\label{thm:(n-1)Pol}
    Assume $\bmu$ is $n$-definite and $\mathrm{rank}~\bgamma_{n-1} = k-\ell$ with $0\leq\ell<k$, where $\bgamma_{n-1}$ is defined in \eqref{eq:MonicOrthogCond}. Let $\{\vec P_j(\diamond,\bmu)\}_{j=0}^{n-2}$ denote the orthonormal polynomials associated with $\bmu$, with recurrence coefficients $\{\vec A_j(\bmu)\}_{j=0}^{n-3}$ and $\{\vec B_j(\bmu)\}_{j=0}^{n-3}$. Consider the polynomial
    \begin{align}\label{eq:Pdef}
        \vec P(x):=\vec P(x,\bmu) = x\vec P_{n-2}(x,\bmu)-\vec P_{n-2}(x,\bmu)\vec A_{n-2}-\vec P_{n-3}(x,\bmu)\vec B_{n-3}^*, 
    \end{align}
    where
    \begin{align*}
        \vec A_{n-2} := \vec A_{n-2}(\bmu) =  \langle \vec P_{n-2}(\diamond,\bmu),\diamond,\vec P_{n-2}(\diamond,\bmu) \rangle_\bmu,
    \end{align*}
    and define $\vec B_{n-2}:=\vec B_{n-2}(\bmu)\in \mathbb C^{(k-\ell)\times k}$ as the unique matrix in row echelon form with positive pivots satisfying $\langle\vec P,\vec P \rangle_\bmu=\vec B_{n-2}^*\vec B_{n-2}$. Then, the polynomial
    \begin{align}\label{eq:Pnm1Def}
        \vec P_{n-1}(x) := \vec P_{n-1}(x,\bmu) =  \vec P(x,\bmu) \vec B_{n-2}^\dagger \in \mathbb C^{k\times (k-\ell)},
    \end{align}
    satisfies
    \begin{align}\label{eq:Pnm1OrthogCond}
        \langle \vec P_{n-1},\vec P_j \rangle_\bmu = \vec 0 \quad \text{for }j<n-1, \quad \langle \vec P_{n-1},\vec P_{n-1}\rangle_\bmu = \vec I_{k-\ell}, \quad \text{and} \quad \langle \vec P_{n-1},\diamond \vec P_{n-2}(\diamond) \rangle_\bmu = \vec B_{n-2}.
    \end{align}
\end{theorem}
\begin{proof}
    We first verify that $\vec P_{n-1}$ is well-defined. From the recurrence of the orthonormal polynomials in \eqref{eq:OrthonormRec}, the leading coefficient of $\vec P$ is given by $\vec B_0^{-1}\dots \vec B_{n-3}^{-1}$, which implies
    \begin{align*}
        \mathrm{rank} \langle \vec P,\vec P\rangle_\bmu = \mathrm{rank}\langle \vec \Pi_{n-1},\vec \Pi_{n-1}\rangle_\bmu = k-\ell.
    \end{align*}
    Lemma~\ref{lem:REF} then guarantees the existence and uniqueness of $\vec B_{n-2}$. Since $\vec B_{n-2}$ has full row rank, it admits a right inverse, and therefore $\vec P_{n-1}$ in \eqref{eq:Pnm1Def} is well-defined. Now, recall that 
    \begin{align*}
        \vec B_{n-3} = \langle \vec P_{n-2}, \diamond\vec P_{n-3}(\diamond) \rangle_\bmu \quad \text{and} \quad \vec A_{n-2} = \langle \vec P_{n-2},\diamond\vec P_{n-2}(\diamond) \rangle_\bmu,
    \end{align*}
    so $\langle \vec P,\vec P_j\rangle_\bmu = \vec 0$ for $j<n-1$, and the orthonormality conditions in \eqref{eq:Pnm1OrthogCond} follow from the definition of $\vec B_{n-2}$. Finally, the identity $\vec B_{n-2} = \langle \vec P_{n-1}, \diamond \vec P_{n-2}(\diamond) \rangle_\bmu$ follows from 
    \begin{align*}
        \langle \vec P,x\vec P_{n-2} \rangle_\bmu = \langle \vec P,\vec P \rangle_\bmu = \vec B_{n-2}^*\vec B_{n-2},
    \end{align*}
    upon left multiplication by $(\vec B_{n-2}^\dagger)^*$.
\end{proof}

\begin{remark}
    For consistency with \eqref{eq:RecCoefForm}, we extend the definition of $\vec A_j$ to $j=n-1$ by setting
    \begin{align*}
        \vec A_{n-1} := \vec A_{n-1}(\bmu) = \langle \vec P_{n-1}(\diamond,\bmu),\diamond\,\vec P_{n-1}(\diamond,\bmu)\rangle_\bmu \in \mathbb C^{(k-\ell)\times (k-\ell)}
    \end{align*}
    This coefficient, together with $\vec A_{n-2}$ and $\vec B_{n-2}$ in Theorem~\ref{thm:(n-1)Pol}, will be important for the inverse spectral analysis carried out in Section~\ref{subsec:InvSpectMap}.
\end{remark}

\section{The spectral map and its inverse for Banded Hermitian matrices}\label{sec:spectralmap}

\subsection{The class \texorpdfstring{$\JkN$}{JkN} and its spectral properties}

In this section, we study the class of banded Hermitian matrices introduced in \eqref{eq:BandedMatDef} and their connection to matrix-valued measures. We first formalize the set of all such matrices and introduce notation that will be used throughout the paper.




\begin{definition}\label{def:JkN}
We denote by $\JkN$ the set of all $N\times N$ Hermitian matrices admitting a block tridiagonal representation of the form \eqref{eq:BandedMatDef} with $n=\lceil N/k\rceil$.  For each $\vec J\in \JkN$, we denote the diagonal and sub-diagonal blocks of $\vec J$ respectively by 
\begin{align*}
    \vec A_j:=\vec A_j(\vec J), \quad j=0,\dots,n-1, \quad \text{and} \quad \vec B_j = \vec B_j(\vec J), \quad j=0,\dotsm,n-2.
\end{align*}
These blocks satisfy the following conditions, with $\ell=nk-N$:
\begin{enumerate}
    \item $\vec A_0,\dots,\vec A_{n-2}\in \mathbb C^{k\times k}$ and $\vec A_{n-1}\in \mathbb C^{(k-\ell)\times (k-\ell)}$ are Hermitian,
    \item $\vec B_0,\dots,\vec B_{n-3}\in \mathbb C^{k\times k}$ and $\vec B_{n-2}\in \mathbb C^{(k-\ell)\times k}$ have full rank,
    \item $\vec B_0, \dots, \vec B_{n-3}$ are upper triangular with strictly positive diagonal entries,
    \item $\vec B_{n-2}$ is in row echelon form with positive leading entries.
\end{enumerate}
\end{definition}

The following result describes eigenvectors corresponding to repeated eigenvalues and the multiplicity of these eigenvalues.

\begin{lemma}\label{lem:LinIndep}
    Let $\vec J \in \JkN$ and suppose $\vec J \vec q_j = \lambda_j \vec q_j$, $j = 1,2,\ldots,N$, where $\{\vec q_j\}_{j=1}^N$ is an orthonormal basis and $\lambda_j \leq \lambda_{j+1}$.  If $\lambda_\ell = \lambda_j$ for $\ell = j,\ldots,j+p$, then\footnote{We use the notation $\vec x_{i:j}$ to denote the entries of $\vec x$ from index $i$ through $j$.}
    \begin{align*}
        (\vec q_j)_{1:k}, (\vec q_{j+1})_{1:k}, \ldots, (\vec q_{j+p})_{1:k},
    \end{align*}
    are linearly independent. As a result, every eigenvalue of $\vec J$ has multiplicity at most $k$.
\end{lemma}
\begin{proof}
    Suppose, for contradiction, that the vectors $(\vec q_j)_{1:k}, \ldots, (\vec q_{j+p})_{1:k}$ are linearly dependent, then there exists a linear combination that gives an eigenvector $\vec v$ for $\lambda_j$ that has its first $k$ entries being all zeros, i.e. 
    \begin{align*}
        \vec v = \begin{bmatrix} \vec 0^T & \vec v_1^T& \cdots~ & \vec v_{n-1}^T \end{bmatrix}^T \in \mathbb C^N.
    \end{align*}
    Using the recurrence implied by $(\vec J - \lambda_j) \vec v = \vec 0$, we find
    \begin{align*}
        \vec A_0 \vec 0 + \vec B^*_0 \vec v_1 &= \vec 0 \quad \Rightarrow \quad \vec v_1 = \vec 0,\\
        &\vdots\\
         \vec B_{j-1} \vec 0 + \vec A_j \vec 0 + \vec B_j^* \vec v_{j+1} &= \vec 0 \quad \Rightarrow \quad \vec v_{j+1} = \vec 0,\\
         &\vdots\\
         \vec B_{n-3} \vec 0 + \vec A_{n-2} \vec 0 + \vec B_{n-2}^* \vec v_{n-1} &= \vec 0 \quad \Rightarrow \quad \vec v_{n-1} = \vec 0.
    \end{align*}
    The last equality follows from the fact that $\vec B_{n-2}^*$ has a left inverse. This implies that 
    \begin{align*}
        \vec q_j, \vec q_{j+1}, \ldots, \vec q_{j+p},
    \end{align*}
    are linearly dependent, a contradiction. As a direct consequence, if an eigenvalue $\lambda_j$ had multiplicity $p > k$, the corresponding vectors $(\vec q_j)_{1:k}, \ldots, (\vec q_{j+p-1})_{1:k}$ would be linearly dependent, which is impossible. Therefore, the multiplicity of any eigenvalue cannot exceed $k$.
\end{proof}


\subsection{Matrix polynomial maps}

Let $\vec P(x) = \sum_{j=0}^p \Phi_j\,x^j$ be a matrix polynomial with coefficients $\Phi_j \in \mathbb C^{k\times k}$. For $\vec J\in \mathbb C^{N\times N}$ and $\vec E\in\mathbb C^{N\times k}$, we define the matrix polynomial map
\begin{align}\label{eq:PolMapDef}
    \vec P(\vec J)\circ\vec E := \sum_{j=0}^p \vec J^j\,\vec E\,\Phi_j \in \mathbb C^{N\times k}.
\end{align}
This notation was first introduced in \cite{kent_chebyshev_1989} and has been used in \cite{casulli_efficient_2024,frommer_block_2020,simoncini_convergence_1996,simoncini_ritz_1996,elsworth_block_2020} to analyze block Krylov subspaces. An analogous definition applies to vector-valued polynomials, i.e., when $\Phi_j \in \mathbb C^k$. We collect the elementary properties of the polynomial map in the following lemma.
\begin{lemma}\label{lem:PolMapProp}
    Let $\vec P,\vec Q$, and $\vec R$ be matrix polynomials and let $j$ be a non-negative integer. The polynomial map, defined in \eqref{eq:PolMapDef}, satisfies the following relations:
    \begin{enumerate}
        \item if $\vec P(x) = \vec Q(x)+\vec R(x)$, then $\vec P(\vec J)\circ \vec E = \left(\vec Q(\vec J)\circ \vec E\right) + \left(\vec R(\vec J) \circ \vec E\right)$,
        \item if $\vec P(x) = x^j \,\vec Q(x)$, then $\vec P(\vec J)\circ \vec E = \vec J^j\,\left(\vec Q(\vec J)\circ \vec E\right)$,
        \item if $\vec P(x) = \vec Q(x)\, \vec C$ for some $\vec C\in \mathbb C^{k\times k}$, then $\vec P(\vec J)\circ \vec E = \left(\vec Q(\vec J) \circ \vec E\right)\,\vec C$.
    \end{enumerate}
\end{lemma}

Throughout most of this document, $\vec E$ is chosen as a block selection matrix $\vec E_j$ defined by
\begin{align}\label{eq:SelMat}
    \vec E_j := \vec E_{j}^{N,k} = \begin{bmatrix}
        \vec 0_{k\times k} &\cdots & \vec 0_{k\times k}& \vec I_k & \vec 0_{k\times k} & \cdots & \vec 0_{k\times (k-\ell)}
    \end{bmatrix}^T \in \mathbb C^{N\times k} \quad j = 1,\dots,n-1,
\end{align}
where $\ell = nk-N$ and $n=\lceil N/k\rceil$. In other words, all blocks of $\vec E_j$ are zero except for the $j$-th block of size $k\times k$, which is the identity. 
For simplicity, the superscripts will be omitted when the dimensions are clear from context.

\begin{lemma}\label{lem:PowersJ}
    Consider $\vec J \in \JkN$ and let $\{\vec B_j\}_{j=0}^{n-2}$ denote its subdiagonal blocks. For any integer $1 \le j \le \lceil N/k\rceil-1$, we have
    \begin{align}\label{eq:PowersJ}
    \vec E_{j+1}^* \vec J^j \vec E_1 = \vec B_{j-1} \vec B_{j-2} \cdots \vec B_0.
    \end{align}
\end{lemma}
\begin{proof}
    For $j=1$, the statement is trivially true. Now, suppose that \eqref{eq:PowersJ} holds for $j-1$, then
    \begin{align*}
        \vec J^j \vec E_1 = \vec J \vec J^{j-1} \vec E_1  = \vec J \begin{bmatrix}
            \vec X_1^T & \vec X_2^T & \dots & \vec X_{j}^T & \vec 0^T & \cdots & \vec 0^T
        \end{bmatrix}^T, \quad \text{with} \quad \vec X_j = \vec B_{j-2} \cdots \vec B_0.
    \end{align*}
    This implies that $\vec E_{j+1}^* \vec J^j \vec E_1 = \vec B_{j-1} \vec X_{j}$.
\end{proof}


\begin{lemma}\label{lem:OpsFromBanded}
    Consider $\vec J \in \JkN$ with associated blocks $\{\vec A_j\}_{j=0}^{n-1}$ and $\{\vec B_j\}_{j=0}^{n-2}$ as in Definition~\ref{def:JkN}. Let $\vec P_{-1}(x)= \vec 0_k$ and $\vec P_0(x)=\vec I_k$, and define a sequence of matrix polynomials $\{\vec P_j\}_{j=1}^{n-1}$ recursively by
    \begin{align}\label{eq:PolRec}
        \vec P_{j}(x) = \vec Q_{j}(x)  \vec B_{j-1}^\dagger, \qquad \vec Q_{j}(x) = 
        x \vec P_{j-1}(x) - \vec P_{j-1}(x)\vec A_{j-1} - \vec P_{j-2}(x)\vec B^*_{j-2}, 
        \qquad j=1,\dots,n-1,
    \end{align}
    where $\vec B_{j-1}^\dagger$ denotes the right inverse of $\vec B_{j-1}$ and $\vec B_{-1}=\vec I_k$. Then, for each $j=0,\dots,n-1$, we have
    \begin{align}\label{eq:OpsFromBandedProp}
        \vec P_j(\vec J)\circ \vec E_1  = \vec E_{j+1}.
    \end{align}
\end{lemma}

\begin{proof}
    For $j=0$, the claim is trivially satisfied. Next, for $j=1,\dots,n-1$, using Lemma~\ref{lem:PolMapProp} together with \eqref{eq:PolRec}, we find
    \begin{align*}
        \vec P_j(\vec J)\circ\vec E_1 = \vec J \left( \vec P_{j-1}(\vec J)\circ\vec E_1 \right)\vec B_{j-1}^\dagger - \left( \vec P_{j-1}(\vec J)\circ\vec E_1 \right)\vec A_{j-1}\vec B_{j-1}^\dagger - \left( \vec P_{j-2}(\vec J)\circ\vec E_1 \right)\vec B_{j-2}^* \vec B_{j-1}^\dagger.
    \end{align*}
    From the definition of $\vec J$ in \eqref{eq:BandedMatDef}, we have
    \begin{equation}\label{eq:ERec}
        \begin{aligned}
        \vec J \vec E_{j} = \vec E_{j-1}\vec B_{j-2}^*+\vec E_{j}\vec A_{j-1} + \vec E_{j+1}\vec B_{j-1}, \quad 1\leq j\leq n-1.
        \end{aligned}
    \end{equation}
    Multiplying \eqref{eq:ERec} on the right by $\vec B_{j-1}^\dagger$, we see that $\vec E_{j+1}$ satisfies the same recurrence as $\vec P_j(\vec J)\circ\vec E_1$, which completes the proof.
\end{proof}

\subsection{The spectral map}

We define the spectral map $\varphi:  \JkN \to \mathcal{P}_k(\mathbb R)$ by
\begin{align}\label{eq:SpectMapDef1}
    \varphi(\vec J) = \sum_{j=1}^N \vec v_j \vec v_j^* \delta_{\lambda_j},
\end{align}
where $\{\lambda_j\}_{j=1}^N$ are the eigenvalues of $\vec J$ and $\vec v_j $ is the first $k$ components of the $j$th normalized eigenvector. Equivalently, $\varphi(\vec J)$ can be written as
\begin{align}\label{eq:SpectMapDef2}
    \varphi(\vec J) = \sum_{j=1}^m \vec V_j\vec V_j^* \delta_{x_j},
\end{align}
where $\{x_j\}_{j=1}^m$ are the distinct eigenvalues of $\vec J$, and $\vec V_j$ is the matrix formed by the first $k$ rows of the eigenvector matrix associated with $x_j$.
This definition does not depend on the choice of eigenbasis. If $\vec{\tilde V}_j$ corresponds to a different choice of eigenvectors for $x_j$, then there exists unitary matrix $\vec Q_j$ such that $\tilde{\vec V}_j = \vec V_j \vec Q_j$, and therefore $\vec V_j \vec V_j^* = \vec{\tilde V}_j \vec{\tilde V}_j^*$. It is also worth noting that Lemma~\ref{lem:LinIndep} implies each $\vec V_j \vec V_j^*$ has rank $n_j \le k$, and the total rank satisfies $\sum_{j} n_j = N$ and $\sum_{j} \vec V_j \vec V_j^* = \vec I_k$. 

\begin{remark}\label{remark:HerglotzDef}
    By Herglotz's representation theorem\cite{MatrixHerglotz}, the spectral map $\varphi(\vec J)=\bmu$ admits an equivalent an equivalent characterization as the unique probability measure satisfying
    \begin{align*}
        \vec I_{N\times k}^*\, (\vec J-z)^{-1} \,\vec I_{N\times k} = \int_{\mathbb R}\frac{\bmu(\mathrm{d}x)}{x-z} \quad \text{for }\mathrm{Im}\,z>0.
    \end{align*}
    This formulation is standard in the spectral theory of infinite matrices and operators, and reduces to \eqref{eq:SpectMapDef1} in the finite case.
\end{remark}

\begin{lemma}\label{lem:SpectMapMoments}
    If $\vec J \in \JkN$, then the moments of $\bmu=\varphi(\vec J)$ satisfy
    \begin{align*}
    \int_{\mathbb R} x^i \bmu(\mathrm d x) = {\vec E_1}^*\vec J^i\vec E_1\quad \text{for }i\geq 0.
    \end{align*}
    Moreover, for polynomials $\vec P,\vec Q$, we have
    \begin{align*}
        \langle \vec P, \vec Q \rangle_\bmu = (\vec P(\vec J)\circ\vec E_1)^*(\vec Q(\vec J)\circ\vec E_1).
    \end{align*}
\end{lemma}
\begin{proof}
    The first identity follows directly from the definition of the spectral measure. Specifically, let $\vec J = \vec U \Lambda \vec U^*$ denote the eigendecomposition of the banded matrix, then
    \begin{align*}
        \int_{\mathbb R} x^i \bmu(\mathrm d x) = \sum_{j=1}^N \lambda_j^i \vec v_j \vec v_j^* = \vec E_1^* \vec U \Lambda^i\vec U^* \vec E_1 =  \vec E_1^*\vec J^i\vec E_1.
    \end{align*}
    Now, suppose that $\vec P(x) = \sum_{i=0}^p \vec C_ix^i$ and $\vec Q(x) = \sum_{i=0}^q \vec D_i x^i$, then 
    \begin{align*}
        (\vec P(\vec J)\circ\vec E_1)^* (\vec Q(\vec J)\circ\vec E_1) = \sum_{i=0}^p\sum_{j=0}^q \vec C_i^* \left( {\vec E_1}^*\vec J^{i+j}\vec E_1 \right)\vec D_j = \sum_{i=0}^p\sum_{j=0}^q \vec C_i^* \left( \int_{\mathbb R}x^{i+j}\bmu(\mathrm dx) \right)\vec D_j = \langle \vec P, \vec Q \rangle_\bmu.
    \end{align*}
\end{proof}

Theorem~\ref{thm:Nd} shows that the quasi-inner product induced by the measure $\bmu = \varphi(\vec J)$, where $\vec J \in \mathcal{J}_{k,N}$ with $N = kn - \ell$, defines a norm on the space of polynomials of degree at most $n-2$, implying that $\bmu$ is $n$-definite. Moreover, there exist exactly $\ell$ independent vector polynomials of degree $n-1$ with zero norm. These results are fundamental for the identifying the range of the spectral map in Section~\ref{subsec:InvSpectMap}.

\begin{theorem}\label{thm:Nd}
    Let $\bmu = \varphi(\vec J)$, where $\vec J \in \mathcal{J}_{k,N}$ with $N = kn - \ell$, $0 \leq \ell < k$, and
    \begin{align}\label{eq:NdDef}
        \mathcal{N}_d = \left\{ 
            \vec p(x) = \sum_{j=0}^d \vec c_j x^j,~ \vec c_j \in \mathbb C^k: 
            \langle \vec p, \vec p \rangle_\bmu = 0 
        \right\}.
    \end{align}
    Then the dimension of $\mathcal{N}_d$ satisfies
    \begin{align}\label{eq:NdDimEq}
        \dim \mathcal{N}_d =
        \begin{cases}
            0, & \text{if } d < n - 1, \\
            \ell, & \text{if } d = n-1.
        \end{cases}
    \end{align}
\end{theorem}
\begin{proof}
    Suppose first that $d < n - 1$, and let $\vec p(x) = \sum_{j=0}^d \vec c_j x^j$ with $\vec c_d \neq \vec 0$. Assume $\langle \vec p, \vec p \rangle_\bmu = 0$, then by Lemma~\ref{lem:SpectMapMoments} we have $(\vec p(\vec J)\circ\vec E_1)^*(\vec p(\vec J)\circ\vec E_1) = 0$, which implies
    \begin{align}\label{eq:ContradicEqForDef}
        \vec J^d \vec E_1 \vec c_d 
        + \sum_{j=0}^{d-1} \vec J^j \vec E_1 \vec c_j = \vec 0.
    \end{align}
    Multiplying \eqref{eq:ContradicEqForDef} on the left by $\vec E_{d+1}^*$ and using Lemma~\ref{lem:PowersJ}, we get $\vec B_{p-1}\dots\vec B_0 \vec c_d = \vec 0$, which contradicts the fact that $\vec B_{p-1}\dots\vec B_0$ is nonsingular.

    Next, assume $\dim \mathcal{N}_{n-1} = r > \ell$. Then there exist linearly independent polynomials $\{ \vec p_i(x) \}_{i=1}^r$ such that
    \begin{align*}
        \vec p_i(x) = \sum_{j=1}^{n-1} \vec c_j^{(i)} x^j \in \mathbb C^k,\quad \vec c_{n-1}^{(i)} \neq \vec 0, \quad   \langle \vec p_i, \vec p_i \rangle_\bmu = 0.
    \end{align*}
    We note that $\langle\vec p_i,\vec p_i \rangle_\bmu = 0$ is equivalent to
    \begin{align}\label{eq:ZeroNormEquiv}
        \vec V_j^* \vec p_i(x_j) = \vec 0 \qquad  \text{for } j=1,\dots,m,
    \end{align}
    where $\bmu = \sum_{j=1}^m \vec V_j \vec V_j^* \delta_{x_j}$.
    We also claim that the leading coefficients $\{\vec c_{n-1}^{(i)}\}_{i=1}^r$ are linearly independent. Indeed, if there exists nonzero constants $\alpha_1,\dots,\alpha_r$ such that $\sum_{i=1}^r \alpha_i \vec c_{n-1}^{(i)} = \vec 0$, then the polynomial $\vec p(x) = \sum_i\alpha_i \vec p_i(x)$ is of degree $n-2$ and satisfies $\langle \vec p,\vec p \rangle_\bmu = 0$ by \eqref{eq:ZeroNormEquiv}, contradicting $\mathrm{dim}~ \mathcal{N}_{n-2}=0$ because $\vec p \neq \vec 0$ by linear independence. Now, using Lemmas~\ref{lem:PowersJ} and \ref{lem:SpectMapMoments} again, we have $\vec B_{n-2}\dots\vec B_0 \vec c_{n-1}^{(i)} = \vec 0$ for $i = 1,\dots,r$. Since $\vec B_{n-1}\dots\vec B_0$ is nonsingular, we find $\dim \ker(\vec B_{n-2}) = r > \ell$, which contradicts the assumption that $\vec B_{n-2} \in \mathbb{C}^{(k-\ell)\times k}$ has full rank, and therefore $r \leq \ell$.

    To establish $r=\ell$, consider the polynomial $\vec Q_{n-1}$ defined in Lemma~\ref{lem:OpsFromBanded} and note that 
    {\small
    \begin{align}\label{eq:RankPTilde}
        \langle \vec Q_{n-1},\vec Q_{n-1} \rangle_\bmu &= \langle \diamond\vec P_{n-2}(\diamond)-\vec P_{n-2}\vec A_{n-2}-\vec P_{n-3}\vec B_{n-3}^*, \diamond\vec P_{n-2}(\diamond)-\vec P_{n-2}\vec A_{n-2}-\vec P_{n-3}\vec B_{n-3}^* \rangle_\bmu = \vec B_{n-2}^*\vec B_{n-2},
    \end{align}
    }
    where the second equality follows from Lemma~\ref{lem:OpsFromBanded} and Lemma~\ref{lem:SpectMapMoments}.
    Hence $\mathrm{rank}~\langle \vec Q_{n-1}, \vec Q_{n-1} \rangle_\bmu = k-\ell$, and there exists linearly independent vectors
    $\{\vec v_i\}_{i=1}^{\ell}$ such that
    \begin{align*}
        \langle \vec p_i, \vec p_i \rangle = 0,
        \qquad \vec p_i(x) := \vec Q_{n-1}(x)\vec v_i, \quad \text{for } i=1,\dots \ell.
    \end{align*}
    Since the leading coefficient of $\vec Q_{n-1}$, given by $\vec B_0^{-1}\cdots \vec B_{n-3}^{-1}$, is nonsingular, the vector polynomials 
    $\{\vec p_i\}_{i=1}^\ell$ are linearly independent and have nonzero leading coefficients.  
    We conclude that $\{\vec p_i\}_{i=1}^\ell \subset \mathcal{N}_{n-1}$ which implies $\mathrm{dim} ~ \mathcal{N}_{n-1} \geq \ell$, and therefore $\mathrm{dim} ~ \mathcal{N}_{n-1} = \ell$.
\end{proof}

Theorem~\ref{thm:Nd} admits an equivalent formulation in terms of matrix polynomials and the structure of the spectral measure, as presented in Theorem~\ref{thm:NdEquiv}.
Statement~\ref{item:Nd-rank} in Theorem~\ref{thm:NdEquiv} shows that when $N = kn - \ell$, the norms of the matrix orthogonal polynomials are of full rank for all polynomials of degree up to $n-2$, and of rank $k - \ell$ for those of degree $n-1$.
Statement~\ref{item:Nd-meas} describes the conditions on the support and weights of a measure required for the existence of such matrix orthogonal polynomials.

\begin{theorem}\label{thm:NdEquiv}
    Let $\vec J \in \mathcal{J}_{k,N}$ with $N = kn - \ell$ and $0 \leq \ell < k$, and let $\bmu = \varphi(\vec J)=\sum_{j=1}^m \vec V_j\vec V_j^* \delta_{x_j}$ be the associated spectral measure, as defined in \eqref{eq:SpectMapDef2}. The following statements are equivalent:
    \begin{enumerate}[leftmargin=*, label=(\arabic*)]
    \item\label{item:Nd-dim} $\mathcal{N}_d$ defined in \eqref{eq:NdDef} satisfies \eqref{eq:NdDimEq}.
    \item\label{item:Nd-rank} 
    For every monic matrix polynomial $\vec \Pi$ of degree $d$, we have
    \begin{align*}
      \mathrm{rank}~\langle \vec \Pi,\vec \Pi\rangle_{\bmu} \geq
      \begin{cases}
        k, & d < n-1,\\
        k-\ell, & d = n-1,
      \end{cases}
    \end{align*}
    and, in particular,
    \begin{align*}
      \mathrm{rank}~\langle \vec \Pi_{n-1}, \vec \Pi_{n-1}\rangle_\bmu = k-\ell
    \end{align*}
    where $\vec\Pi_{n-1}=\vec \Pi_{n-1}(\diamond,\bmu)$ denotes the monic orthogonal polynomial of degree $n-1$.
    \item\label{item:Nd-meas} Let 
        \begin{align*}
            \vec X = \mathrm{diag}(\underbrace{x_1,\dots,x_1}_{n_1},\underbrace{x_2,\dots,x_2}_{n_2},\ldots,\underbrace{x_m,\dots,x_m}_{n_m}), \quad \vec V^* = \begin{bmatrix} \vec V_1 & \cdots & \vec V_m \end{bmatrix}, \quad \vec V \in \mathbb C^{N \times k},
        \end{align*}
        where $n_j = \mathrm{rank}~\vec V_j$, then the matrix
        \begin{align*}
            \vec M_d(\vec V, \vec X) := \begin{bmatrix} \vec V  & \vec X \vec V &  \cdots &  \vec X^d \vec V \end{bmatrix} \in \mathbb C^{N\times (d+1)k}
        \end{align*}
        has full rank, that is,
        \begin{align*}
        \mathrm{rank}~\vec M_d(\vec V,\vec X) = \min\{(d+1) k, N\}.
        \end{align*}
    \end{enumerate}
\end{theorem}
\begin{proof}
    The proof proceeds in two steps. We first show that \ref{item:Nd-dim} and \ref{item:Nd-rank} are equivalent, and then establish the equivalence between \ref{item:Nd-dim} and \ref{item:Nd-meas}.
    
    \medskip
    \noindent\underline{Step 1: Equivalence of \ref{item:Nd-dim} and \ref{item:Nd-rank}.}
    First, suppose that \ref{item:Nd-dim} holds. If $\vec \Pi$ is a monic polynomial of degree $d$, then $\mathrm{rank}\,\langle \vec \Pi, \vec \Pi \rangle_\bmu = k$ for $d < n-1$; otherwise, there would exist a nonzero vector $\vec v$ such that $\vec p(x) = \vec \Pi(x)\vec v$ satisfies $\langle \vec p, \vec p \rangle_\bmu = 0$, contradicting \ref{item:Nd-dim}. Next, assume that $\mathrm{rank}\,\langle \vec \Pi, \vec \Pi \rangle_\bmu = k - r$ for $d = n-1$ with $r > \ell$, then there exists linearly independent vectors $\{\vec v_i\}_{i=1}^r$ such that
    \begin{align*}
        \langle \vec p_i, \vec p_i \rangle_\bmu = 0, \quad
        \vec p_i(x) := \vec \Pi(x)\vec v_i, \quad i = 1,\dots,r.
    \end{align*}
    Since $\vec \Pi$ is monic, each $\vec p_i$ is a vector polynomial of degree $n-1$ and the set $\{\vec p_i\}_{i=1}^r$ is linearly independent, implying $\dim \mathcal{N}_{n-1} = r \geq \ell$, again contradicting \ref{item:Nd-dim}. 
    To show that the rank condition holds for $\vec \Pi_{n-1}$, consider the polynomial $\vec Q_{n-1}$ defined in Lemma~\ref{lem:OpsFromBanded}. Since $\vec \Pi_{n-1}(x) = \vec Q_{n-1}(x)\vec B_{n-3}\cdots\vec B_{0}$, we have
    \begin{align*}
        \mathrm{rank}~\langle \vec \Pi_{n-1}, \vec \Pi_{n-1} \rangle_\bmu
        = \mathrm{rank}~\langle \vec Q_{n-1}, \vec Q_{n-1} \rangle_\bmu
        = k - \ell,
    \end{align*}
    where the second equality follows from \eqref{eq:RankPTilde}.  

    Now suppose that \ref{item:Nd-rank} holds, and assume that there exists a non-trivial polynomial $\vec p\in \mathbb C^k$ of degree $d$ with $d<n-1$ such that $\langle \vec p, \vec p \rangle_\bmu = 0$. Let $\vec P \in\mathbb C^{k\times k}$ be a matrix polynomial of degree $d$ whose leading coefficient $\vec C_d$ is invertible and whose first column is $\vec p$. Define $\vec \Pi = \vec P \vec C_d^{-1}$ and $\vec v = \vec C_d\vec e_1$, then
    \begin{align*}
        \vec v^* \langle \vec \Pi, \vec \Pi \rangle_\bmu \vec v = \langle \vec p, \vec p \rangle_\bmu = 0.
    \end{align*}
    Hence $\langle \vec \Pi, \vec \Pi \rangle_\bmu$ is rank deficient, which leads to a contradiction. Since $\mathrm{rank}\,\langle \vec \Pi_{n-1}, \vec \Pi_{n-1} \rangle_\bmu = k - \ell$, we have $\dim \mathcal{N}_{n-1} \ge \ell$, and it remains to prove that $\dim \mathcal{N}_{n-1} \le \ell$. If this were false, there would exist $r > \ell$ linearly independent vector polynomials $\{\vec p_i\}_{i=1}^r$ of degree $n-1$ such that $\langle \vec p_i, \vec p_i \rangle_\bmu = 0$ for $i = 1, \dots, r$. Since $\mathrm{dim}~\mathcal{N}_{n-2}=0$, the leading coefficients of $\{\vec p_i\}_{i=1}^r$ are linearly independent. As a result, there exists a matrix polynomial $\vec P$ of degree $d$ whose leading coefficient $\vec C_{n-1}$ is invertible and whose first $r$ columns are given by $\vec p_1,\dots, \vec p_r$. Setting $\vec \Pi = \vec P \vec C_d^{-1}$ and $\vec v_i = \vec C_{n-1} \vec e_i$ for $i=1,\dots,r$, we find
    \begin{align*}
        \vec v_i^* \langle \vec \Pi, \vec \Pi \rangle_\bmu \vec v_i 
        = \langle \vec p_i, \vec p_i \rangle_\bmu 
        = 0,
        \qquad i = 1, \dots, r,
    \end{align*}
    so $\mathrm{rank}~\langle \vec \Pi, \vec \Pi \rangle_\bmu<k-\ell$, contradicting \ref{item:Nd-rank}.

    \medskip
    \noindent\underline{Step 2: Equivalence of \ref{item:Nd-dim} and \ref{item:Nd-meas}.}
    For a polynomial $\vec p(x) = \sum_{j=0}^{d} \vec c_j x^j$, we have
    \begin{align*}
        \begin{bmatrix}
            \vec V_1^* \vec p(x_1) \\ \vec V_2^* \vec p(x_2) \\ \vdots \\ \vec V_m^* \vec p(x_m)
        \end{bmatrix} = \begin{bmatrix}
            \vec V_1^*  & x_1 \vec V_1^* & \cdots & x_1^{d} \vec V_1^* \\ \vec V_2^*  & x_2 \vec V_2^* & \cdots & x_2^{d} \vec V_2^* \\ \vdots & \vdots && \vdots \\
            \vec V_m^* & x_m \vec V_m^* & \cdots & x_m^{d} \vec V_m^* 
        \end{bmatrix}\begin{bmatrix} \vec c_0 \\ \vec c_1 \\ \vdots  \\ \vec c_d \end{bmatrix} = \vec M_d(\vec V, \vec X) \begin{bmatrix} \vec c_0 \\ \vec c_1 \\ \vdots  \\ \vec c_d \end{bmatrix},
    \end{align*}
    which implies $\mathrm{dim}~\mathcal{N}_d = \mathrm{dim}~\mathrm{ker} ~\vec M_d(\vec V,\vec X)$. Thus, the matrix $\vec M_d(\vec V,\vec X)$ has full rank if and only if \ref{item:Nd-dim} is satisfied, completing our proof.
\end{proof}

\begin{remark}
    For the spectral measure $\bmu$ defined in \eqref{eq:SpectMapDef2}, the matrices $\vec V_j$ are defined only up to right multiplication by a unitary matrix. However, for any unitary matrices $\vec Q_j$, setting $\vec{\tilde V}_j = \vec V_j \vec Q_j$ gives
    \begin{align*}
        \vec M_d(\vec{\tilde V}, \vec X) 
        = \mathrm{diag}(\vec Q_1^*,\dots,\vec Q_m^*)\, \vec M_d(\vec V, \vec X).
    \end{align*}
    Therefore, if the condition in Theorem~\ref{thm:NdEquiv}~\ref{item:Nd-meas} is satisfied for one choice $\{\vec V_j\}_{j=1}^m$, it is satisfied for any other choice $\{\vec{\tilde V}_j\}_{j=1}^m$.
\end{remark}

The results concerning the null space $\mathcal{N}_d$ and the rank of matrix polynomials for higher degrees can be easily obtained once the conditions in Theorem~\ref{thm:NdEquiv} are satisfied, as stated in the next corollary.

\begin{corollary}\label{cor:ExtRes}
    Let $\bmu$ be a spectral measure associated with a matrix in $\JkN$ with $N = kn-\ell$.
    For every $d \ge n$, there exists a monic matrix polynomial $\vec{\Pi}$ of degree $d$ such that $\langle \vec{\Pi}, \vec{\Pi} \rangle_{\bmu} = \vec{0}$. As a result, we have
    \begin{align*}
        \mathrm{dim}~\mathcal{N}_d = k(d+1)-N, \qquad \text{and} \qquad \langle \vec \Pi_d,\vec \Pi_d \rangle_\bmu = \vec 0,
    \end{align*}
    where $\vec \Pi_d=\vec \Pi_d(\diamond,\bmu)$ is any monic polynomial of degree $d$ satisfying $\langle \vec \Pi_d,\vec P\rangle_\bmu =\vec 0$ for all matrix polynomials $\vec P$ of degree $p<d$.
\end{corollary}
\begin{proof}
    Theorems~\ref{thm:Nd} and Theorem~\ref{thm:NdEquiv}\,\ref{item:Nd-meas} guarantee that, for $d \geq n$, the matrix $\vec{M}_{d-1}(\vec{V}, \vec{X})$ has linearly independent rows and therefore admits a right inverse $\vec{M}_{d-1}^{\dagger}(\vec{V}, \vec{X})$. Define $\{\vec C_j\}_{j=0}^{d-1}$ such that $\vec C_j \in \mathbb C^{k\times k}$ and
    \begin{align*}
        \begin{bmatrix}
            \vec C_0 \\
            \vec C_1\\
            \vdots \\
            \vec C_{d-1}
        \end{bmatrix} = \vec M_{d-1}^\dagger(\vec V,\vec X) \begin{bmatrix}
            -\vec V_1^* x_1^d \\
            -\vec V_2^* x_2^d \\
            \vdots \\
            -\vec V_m^* x_m^d
        \end{bmatrix},
    \end{align*}
    and consider $\vec \Pi(x) = \vec I_k x^d+\sum_{j=0}^{d-1}\vec C_j x^j$. By construction, we have 
    \begin{align*}
        \begin{bmatrix}
            \vec V_1^* \vec \Pi(x_1)\\
            \vec V_2^* \vec \Pi(x_2)\\
            \vdots\\
            \vec V_m^* \vec \Pi(x_m)
        \end{bmatrix} = \vec M_d(\vec V,\vec X) \begin{bmatrix}
            \vec C_0\\
            \vdots \\
            \vec C_{d-1}\\ 
            \vec I_k
        \end{bmatrix} = \vec M_{d-1}(\vec V,\vec X) \begin{bmatrix}
            \vec C_0 \\
            \vec C_1 \\
            \vdots \\
            \vec C_{d-1}
        \end{bmatrix} + \begin{bmatrix}
            \vec V_1^* x_1^d \\
            \vec V_2^* x_2^d \\
            \vdots \\
            \vec V_m^* x_m^d
        \end{bmatrix},
    \end{align*}
    and therefore
    \begin{align}\label{eq:ZeroCond}
        \vec V_j^* \vec \Pi(x_j) = \vec 0 \quad \text{for}~ j=1,\dots,m,
    \end{align}
    which implies $\langle \vec \Pi, \vec \Pi \rangle_\bmu = \vec 0$.

    Since $\dim \mathcal{N}_{n-1} = \ell$ and each degree increment gives at most $k$ dimensions, we have
    \begin{align*}
        \mathrm{dim}~\mathcal{N}_d \leq \ell+k(d-n+1) = k(d+1)-N.
    \end{align*}
    The columns of $\vec{\Pi}$ are linearly independent elements of $\mathcal{N}_d$ for all $d \ge n$, so the above inequality is in fact an equality. Finally, since $\vec{\Pi}_d$ is a monic orthogonal polynomial of degree $d$, the difference $\vec{\Pi}_d - \vec{\Pi}$ has degree at most $d-1$. Using Lemma~\ref{lem:low_deg_expansion}, together with the orthogonality of $\vec \Pi_d$ and \eqref{eq:ZeroCond}, we have
    \begin{align*}
        \langle \vec \Pi_n-\vec{ \Pi},\vec \Pi_n-\vec{ \Pi} \rangle_\bmu = \langle \vec \Pi_n ,\vec \Pi_n-\vec{ \Pi} \rangle_\bmu - \langle \vec{\Pi},\vec \Pi_n-\vec{ \Pi}\rangle_\bmu = \vec 0.
    \end{align*}
    This implies that $\vec V_j^* \left(\vec \Pi_n(x_j)-\vec{ \Pi}(x_j)\right) = \vec 0$ for $j=1,\dots,m$, and hence $\langle \vec \Pi_n, \vec \Pi_n \rangle_\bmu = \vec 0$.
\end{proof}

In the remainder of this section, we show that each spectral measure corresponds to a unique matrix $\vec J$ within the class $\JkN$. Hence, the spectral map is invertible and the explicit construction of its inverse will be addressed in the next section.

\begin{theorem}\label{thm:Injective}
    The spectral map $\varphi: \JkN \to \mathcal{P}_k(\mathbb R)$, defined in \eqref{eq:SpectMapDef1}, is injective.
\end{theorem}
\begin{proof}
    Let $\vec J \in \JkN$ and suppose that $\bmu = \varphi(\vec J)$. Consider the sequences of polynomials $\{\vec Q_j\}_{j=0}^{n-1}$ and $\{\vec P_j\}_{j=0}^{n-1}$ associated with $\vec J$ as defined in Lemma~\ref{lem:OpsFromBanded}, then
    \begin{align*}
    \langle \vec P_i, \vec P_j \rangle_\bmu
    = (\vec P_i(\vec J)\circ \vec E_1)^* \, (\vec P_j(\vec J)
    \circ\vec E_1)
    = \vec E_{i+1}^* \vec E_{j+1}
    = \vec I  \delta_{ij},
    \end{align*}
    where the first equality follows from Lemma~\ref{lem:SpectMapMoments} and the second from Lemma~\ref{lem:OpsFromBanded}.
    Thus, $\{\vec P_j\}_{j=0}^{n-1}$ are orthonormal polynomials with respect to $\bmu$. 
    Now let $\vec{\hat J} \in \JkN$ be another banded matrix such that $\varphi(\vec{\hat J}) = \bmu$, and let $\{\vec Q_j\}_{j=0}^{n-1}$ and $\{\vec{\hat P}_j\}_{j=0}^{n-1}$ be the associated polynomials from Lemma~\ref{lem:OpsFromBanded}. By the same argument, $\{\vec{\hat P}_j\}_{j=0}^{n-1}$ also form a set of orthonormal polynomials with respect to $\bmu$.

    We now prove that the two families of polynomials coincide. Let $\{\vec A_j\}_{j=0}^{n-1}$, $\{\vec B_j\}_{j=0}^{n-2}$ and $\{\hat{\vec A}_j\}_{j=0}^{n-1}$, $\{\hat{\vec B}_j\}_{j=0}^{n-2}$ denote the diagonal and off-diagonal blocks of $\vec J$ and $\hat{\vec J}$, respectively, and define
    \begin{align*}
        \vec C_j = \vec B_{j-1}\dots \vec B_0, \quad \hat{\vec C}_j = \hat{\vec B}_{j-1}\dots \hat{\vec B}_0, \quad \text{for }j=0,\dots,n-1.
    \end{align*}
    The recurrence relation in Lemma~\ref{lem:OpsFromBanded} implies that $\vec \Pi_j := \vec Q_j \vec C_{j-1}$ and $\hat{\vec \Pi}_j := \hat{\vec Q}_j \hat{\vec C}_{j-1}$ are monic orthogonal polynomials of degree $j$. By Theorem~\ref{Thm:DefiniteProp} and Theorem~\ref{thm:Nd}, monic orthogonal polynomials of degree at most $n-1$ are unique, so $\vec \Pi_j = \hat{\vec \Pi}_j$ and
    \begin{align}\label{eq:EqMonicNorms}
        \langle \vec \Pi_j,\vec \Pi_j \rangle_\bmu = \langle \hat{\vec \Pi}_j, \hat{\vec \Pi}_j \rangle_\bmu, \quad \text{for }j=0,\dots,n-1.
    \end{align}
    Using Lemma~\ref{lem:OpsFromBanded} and Lemma~\ref{lem:SpectMapMoments}, we have
    \begin{align}\label{eq:QjNorm}
        \langle \vec Q_{j}, \vec Q_{j} \rangle_\bmu = \langle \diamond\vec P_{j-1}(\diamond)-\vec P_{j-1}\vec A_{j-1}-\vec P_{j-2}\vec B_{j-2}^*, \diamond\vec P_{j-1}(\diamond)-\vec P_{j-1}\vec A_{j-1}-\vec P_{j-2}\vec B_{j-2}^* \rangle_\bmu = \vec B_{j-1}^*\vec B_{j-1},
    \end{align}
    and similarly $\langle \hat{\vec Q}_j,\hat{\vec Q}_j \rangle_{\bmu} = \hat{\vec B}_{j-1}^* \hat{\vec B}_{j-1}$, for $j=1,\dots, n-1$.
    Combining \eqref{eq:EqMonicNorms} and \eqref{eq:QjNorm} gives
    \begin{align*}
        \vec C_j^*\vec C_j = \hat{\vec C}_j^*\hat{\vec C}_j, \quad \text{for } j=0,\dots,n-1.
    \end{align*}
    Lemma~\ref{lem:REF} shows that $\vec C_j = \hat{\vec C}_j$ and thus $\vec Q_j = \hat{\vec Q}_j$ which implies $\vec P_j = \hat{\vec P}_j$.
    Finally, by another application of Lemma~\ref{lem:OpsFromBanded} and Lemma~\ref{lem:SpectMapMoments}, we have
    {\small
    \begin{align*}
        \vec A_j = (\vec P_j(\vec J)\circ\vec E_1)^* \vec J (\vec P_j(\vec J)
        \circ\vec E_1) = \langle \vec P_j,\diamond\vec P_j(\diamond) \rangle_\bmu= \langle \vec{\hat P}_j, \diamond\vec{\hat P}_j(\diamond)\rangle_\bmu = (\vec{\hat P}_j(\vec{\hat J})\circ\vec E_1)^* \vec {\hat J}(\vec{\hat P}_j(\vec{\hat J})\circ\vec E_1)= \vec{\hat A}_j,
    \end{align*}
    }
    and
    {\small
    \begin{align*}
        \vec B_j = (\vec P_{j+1}(\vec J)\circ\vec E_1)^* \vec J (\vec P_j(\vec J)\circ\vec E_1) = \langle \vec P_{j+1},\diamond\vec P_j(\diamond) \rangle_\bmu= \langle \vec{\hat P}_{j+1}, \diamond\vec{\hat P}_j(\diamond)\rangle_\bmu = (\vec{\hat P}_{j+1}(\vec{\hat J})\circ\vec E_1)^* \vec {\hat J} (\vec{\hat P}_j(\vec{\hat J})\circ\vec E_1)= \vec{\hat B}_j,
    \end{align*}
    }
    and therefore $\vec J = \vec{\hat J}$.
    
\end{proof}

\subsection{Inverse spectral map}\label{subsec:InvSpectMap}

Given that the spectral map $\varphi$ is injective, it is natural to determine its range and describe the procedure for constructing its inverse. Motivated by Theorems~\ref{thm:Nd} and~\ref{thm:NdEquiv}, we expect the range of $\varphi$ to consist of the measures belonging to the set
{\small
\begin{align}\label{eq:Mkn}
\mathcal{M}_{k,N} := 
\left\{ \sum_{j=1}^m \vec W_j\delta_{x_j} : 
\sum_{j=1}^m\vec W_j=\vec I_k,\;
\vec W_j\!\ge\!0,\;
\textstyle\sum_{j=1}^m\mathrm{rank}\,\vec W_j=N,\;
\text{\eqref{eq:NdDimEq} holds for } N=kn-\ell,\;0\le\ell<k
\right\}.
\end{align}
}
In this section, we prove that the range of $\varphi$ coincides exactly with $\mathcal{M}_{N,k}$ and define the inverse spectral map, which is closely related to the construction of orthonormal matrix polynomials.

Let $\bmu \in \mathcal{M}_{k,N}$, then by Lemma~\ref{thm:NdEquiv} the sequences of monic orthogonal polynomials $\{\vec\Pi_j(\diamond,\bmu)\}_{j=0}^{n-1}$ and orthonormal matrix polynomials $\{\vec P_j(\diamond,\bmu)\}_{j=0}^{n-2}$ are well defined and are given by the recurrence relations in \eqref{eq:MonicRec} and \eqref{eq:OrthonormRec}. Moreover, the $(n-1)$-th orthonormal polynomial $\vec P_{n-1}$ is uniquely defined, as established in Theorem~\ref{thm:(n-1)Pol}. The recurrence coefficients $\{\vec A_j\}_{j=0}^{n-1}$ and $\{\vec B_j\}_{j=0}^{n-2}$ associated with the orthonormal polynomials $\{\vec P_{j}\}_{j=0}^{n-1}$ define a Hermitian, block tridiagonal matrix
\begin{align}\label{eq:BlockJac}
    \vec J = \vec J(\bmu) =  \begin{bmatrix}
    \vec A_0 & \vec B_0^* & & & & \\
    \vec B_0 & \vec A_1 & \vec B_1^* & & & \\
    & \vec B_1 & \vec A_2 & \vec B_2^* & & \\
    & & \vec B_2 & \vec A_3 & \ddots & \\
    & & & \ddots & \ddots & \vec B_{n-2}^* \\
    & & & & \vec B_{n-2} & \vec A_{n-1}
    \end{bmatrix}\in \mathbb C^{N\times N}, \quad  N = kn-\ell.
\end{align}
By enforcing the normalization from Lemma~\ref{lem:OpsUniqNormalization}, the matrix $\vec J(\bmu)$ belongs to $\JkN$ which shows that the correspondence
\begin{align}\label{eq:InvSpecDef}
\bmu \in \mathcal{M}_{k,N} ~  \overset{\psi}{\mapsto} ~\vec J(\bmu) \in \JkN,
\end{align}
is a well-defined map. 

\begin{lemma}\label{lem:RecMatForm}
    For a measure $\bmu \in \mathcal{M}_{k,N}$ with $N = kn-\ell$, denote by $\vec J = \psi(\bmu)$ the associated banded Hermitian matrix from $\JkN$ and by $\{\vec P_j(\diamond,\bmu)\}_{j=0}^{n-1}$ the corresponding orthonormal polynomials defined in \eqref{eq:OpsNormalization} and \eqref{eq:Pnm1Def}.
    Define 
    \begin{align*}
        \mathcal{P}^*(x):=\mathcal{P}^*(x,\bmu) = \begin{bmatrix}
            \vec P_0(x,\bmu) & \vec P_1(x,\bmu) & \dots & \vec P_{n-2}(x,\bmu) & \vec P_{n-1}(x,\bmu)
        \end{bmatrix},
    \end{align*}
    then $\mathcal{P}(x)$ satisfies
    \begin{align}\label{eq:RecMatForm}
        x \,\mathcal P(x) = \vec J \,\mathcal P(x) + \mathcal R(x),
    \end{align}
    where $\mathcal R^*(x) := \mathcal R(x,\bmu)^*$ is given by
    \begin{align*}
        \mathcal R(x,\bmu)^*= \begin{bmatrix}
            \vec 0_{k\times k}  & \dots & \vec 0_{k\times k} & \vec R_{n-2}(x,\bmu) & \vec R_{n-1}(x,\bmu)
        \end{bmatrix}, \quad \vec R_{n-2}(x)\in \mathbb C^{k\times k}, \quad \vec R_{n-1}(x) \in \mathbb C^{k\times \ell},
    \end{align*}
    and satisfies 
    \begin{align*}
        \langle \vec R_{n-2}(\diamond,\bmu) ,\vec R_{n-2}(\diamond,\bmu) \rangle_\bmu = \vec 0, \qquad \langle \vec R_{n-1}(\diamond,\bmu) ,\vec R_{n-1}(\diamond,\bmu) \rangle_\bmu = \vec 0.
    \end{align*}
\end{lemma}
\begin{proof}
    The recurrence relations satisfied by the orthonormal polynomial $\{\vec P_j\}_{j=0}^{n-2}$ in \eqref{eq:OrthonormRec} imply that the first $n-2$ blocks of $\mathcal R$ are zero. Moreover, from the structure of $\vec J$, we have 
    \begin{align*}
        \vec R_{n-2}(x) = x\, \vec P_{n-2}(x)-\vec P_{n-3}(x) \vec B_{n-3}^*-\vec P_{n-2}(x)\vec A_{n-2}-\vec P_{n-1}(x) \vec B_{n-2},
    \end{align*}
    and
    \begin{align*}
        \vec R_{n-1}(x) = x\,\vec P_{n-1}(x) - \vec P_{n-2}(x)\vec B_{n-2}^*-\vec P_{n-1}(x)\vec A_{n-1}.
    \end{align*}
    Using the definition of $\vec P(x)$ in Theorem~\ref{thm:(n-1)Pol}, we find 
    \begin{align*}
        \vec R_{n-2}(x) = \vec P(x) - \vec P_{n-1}(x)\vec B_{n-2} = \vec P(x)\left(\vec I_k-\vec B_{n-2}^\dagger\vec B_{n-2} \right),
    \end{align*}
    where the last equality follows from $\vec P_{n-1}(x) = \vec P(x) \vec B_{n-2}^\dagger$. Finally, since $\langle \vec P,\vec P\rangle_\bmu = \vec B_{n-2}^*\vec B_{n-2}$, we have
    \begin{align*}
        \langle \vec R_{n-2}, \vec R_{n-2} \rangle_\bmu = (\vec I_k-\vec B_{n-2}^\dagger \vec B_{n-2})^*\langle \vec P,\vec P \rangle_\bmu (\vec I_k-\vec B_{n-2}^\dagger \vec B_{n-2}) = \vec 0.
    \end{align*}
    
    It remains to show that $\vec R_{n-1}(x)$ has zero norm. Recall from Corollary~\ref{cor:ExtRes} that $\langle\vec \Pi_n,\vec \Pi_n \rangle_\bmu=\vec 0$, where $\vec \Pi_n(x)$ is the $n$-th monic orthogonal polynomial. Substituting the three-term recurrence from Theorem~\ref{thm:MonicRec} gives
    \begin{align}\label{eq:NormxPi} 
        \langle \diamond\vec \Pi_{n-1}(\diamond),\diamond\vec \Pi_{n-1}(\diamond) \rangle_\bmu = \bgamma_{n-1}\vec C_{n-1}^2 + \vec D_{n-1}^* \bgamma_{n-2}\vec D_{n-1}, 
    \end{align} 
    where $\bgamma_j$ is defined in \eqref{eq:MonicOrthogCond} and $\vec C_{n-1}, \vec D_{n-1}$ are as stated in Theorem~\ref{thm:MonicRec}.
    To connect $\vec \Pi_{n-1}$ and $\vec P_{n-1}$, we decompose $\bgamma_{n-1}$ as
    \begin{align*}
    \bgamma_{n-1} = \left(\vec U_{n-1} \vec \Lambda_{n-1}^{1/2}\right)\left(\vec U_{n-1} \vec \Lambda_{n-1}^{1/2}\right)^*,
    \end{align*}
    where $\vec \Lambda_{n-1}$ is a diagonal matrix containing the nonzero eigenvalues of $\bgamma_{n-1}$, and $\vec U_{n-1}$ is a matrix whose columns consist of the corresponding eigenvectors. This gives the normalization
    \begin{align}\label{eq:LastNormalization}
        \vec P_{n-1}(x) = \vec \Pi_{n-1}(x) \vec U_{n-1}\vec \Lambda_{n-1}^{-1/2}\vec Q_{n-1},
    \end{align}
    which follows the same structure as the normalizations in \eqref{eq:OpsNormalization}, for instance,
    \begin{align*}
        \vec P_{n-2}(x) = \vec \Pi_{n-2}(x) \bgamma_{n-2}^{-1/2}\vec Q_{n-2}.
    \end{align*}
    Here, $\vec Q_{n-2}$ is the unique unitary matrix determined by Lemma~\ref{lem:OpsUniqNormalization} and $\vec Q_{n-1}$ is the unitary matrix that guarantees that $\vec B_{n-2}$ is in row echelon form with positive pivots. In other words, $\vec Q_{n-2}$ and $\vec Q_{n-1}$ are unitary matrices that satisfy
    \begin{equation}\label{eq:ABrel}
    \begin{aligned}
        \vec A_{n-1} &= \langle \diamond\vec P_{n-1}(\diamond),\vec P_{n-1} \rangle_\bmu\\ 
        &= \vec Q_{n-1}^* \vec \Lambda_{n-1}^{-1/2}\vec U_{n-1}^* \langle \diamond\vec \Pi_{n-1}(\diamond),\vec \Pi_{n-1}\rangle_\bmu \vec U_{n-1}\vec \Lambda_{n-1}^{-1/2} \vec Q_{n-1} \\ 
        &= \vec Q_{n-1}^* \vec \Lambda_{n-1}^{1/2}\vec U_{n-1}^*\vec C_{n-1} \vec U_{n-1}\vec \Lambda_{n-1}^{-1/2} \vec Q_{n-1},\\
        \vec B_{n-2} &= \langle \diamond \vec P_{n-1}(\diamond),\vec P_{n-2} \rangle_\bmu\\
        &= \vec Q_{n-1}^* \vec \Lambda_{n-1}^{-1/2}\vec U_{n-1}^*\langle \diamond\vec \Pi_{n-1}(\diamond),\vec \Pi_{n-2} \rangle_\bmu \bgamma_{n-2}^{-1/2} \vec Q_{n-2} \\
        &= \vec Q_{n-1}^* \vec \Lambda_{n-1}^{-1/2}\vec U_{n-1}^*\vec D_{n-1}^* \bgamma_{n-2}^{1/2}\vec Q_{n-2}.
    \end{aligned}
    \end{equation}
    Combining these identities with \eqref{eq:NormxPi} and \eqref{eq:LastNormalization}, gives
    \begin{equation}\label{eq:xPnm1}
    \begin{aligned}
        \langle \diamond\vec P_{n-1}(\diamond), \diamond \vec P_{n-1}(\diamond) \rangle_\bmu &= \vec Q_{n-1}^* \vec \Lambda_{n-1}^{-1/2}\vec U_{n-1}^*\langle \diamond\vec \Pi_{n-1}(\diamond),\diamond\vec \Pi_{n-1}(\diamond)\rangle_\bmu \vec U_{n-1}\vec \Lambda_{n-1}^{-1/2} \vec Q_{n-1}\\
        &= \vec Q_{n-1}^* \vec \Lambda_{n-1}^{-1/2}\vec U_{n-1}^* \left( \bgamma_{n-1}\vec C_{n-1}^2 + \vec D_{n-1}^* \bgamma_{n-2}\vec D_{n-1} \right) \vec U_{n-1}\vec \Lambda_{n-1}^{-1/2} \vec Q_{n-1} \\
        &= \vec A_{n-1}^2 + \vec B_{n-2} \vec B_{n-2}^*.
    \end{aligned}
    \end{equation}
    A direct computation using \eqref{eq:xPnm1} gives $\langle \vec R_{n-1},\vec R_{n-1} \rangle_\bmu = \vec 0$, which concludes the proof.
\end{proof}

\begin{theorem}\label{thm:InvSpec}
    Let $\bmu \in \mathcal{M}_{k,N}$, then $\bmu = \varphi(\psi(\bmu))$.
\end{theorem}
\begin{proof}
    Let $\bmu \in \mathcal{M}_{k,N}$ and write $\bmu = \sum_{j=1}^m \vec V_j \vec V_j^* \delta_{x_j}$. Using Lemma~\ref{lem:RecMatForm} and evaluating \eqref{eq:RecMatForm} at $x_j$, followed by a multiplication on the right by $\vec V_j$, we find for $\vec J = \psi(\bmu)$,
    \begin{align*}
        x_j \mathcal P(x_j)\vec V_j = \vec J \mathcal P(x_j) \vec V_j \qquad j=1,\dots,m.
    \end{align*}
    Since $\vec V_j$ has full rank $n_j$, the resulting vectors $\{\mathcal P \, \vec V_j \, \vec e_i\}_{i=1}^{n_j}$ are linearly independent, proving that the eigenvalues of $\vec J = \psi(\bmu)$ coincide with $x_j$ and that the first $k$ rows of the corresponding eigenvectors are exactly $\vec V_j$. These quantities uniquely determine the measure associated with $\vec J$ via the mapping $\varphi$, and this measure coincides with $\bmu$, i.e. $\bmu = \varphi(\vec J)$.
\end{proof}

Collecting the results of this section, Theorems~\ref{thm:Nd} and~\ref{thm:NdEquiv} show that $\mathrm{Ran}~\varphi \subset \mathcal{M}_{k,N}$, where $\mathcal{M}_{k,N}$ is precisely the domain on which $\psi$ is well defined. Theorem~\ref{thm:InvSpec} further shows that $\varphi$ is surjective onto $\mathcal{M}_{k,N}$, i.e. $\mathrm{Ran}~\varphi = \mathcal{M}_{k,N}$, and that $\varphi$ is a left-inverse of $\psi$. Theorem~\ref{thm:Injective} establishes that $\varphi$ is injective, and together these properties imply that $\varphi$ is also a right-inverse. Indeed, for any $\vec J \in \JkN$, applying Theorem~\ref{thm:InvSpec} with $\bmu = \varphi(\vec J)$ gives $\varphi(\vec J) = \varphi\left(\psi(\varphi(\vec J))\right)$. By injectivity of $\varphi$, it follows that $\psi(\varphi(\vec J)) = \vec J$. Therefore, we conclude that $\varphi\colon \JkN \to \mathcal{M}_{k,N}$ is a bijection and $\psi = \varphi^{-1}$. As a consequence of Theorems~\ref{thm:NdEquiv}~\ref{item:Nd-meas} and \ref{thm:InvSpec}, we also obtain a full characterization of measures associated with matrices in $\JkN$.

\begin{corollary}\label{cor:NecSuffCond}
    Let $\bmu = \sum_{j=1}^m \vec V_j\vec V_j^* \delta_{x_j}$ where the points $\{x_j\}_{j=1}^m$ are distinct and each $\vec V_j\in \mathbb C^{k}\times n_j$ with $n_j\leq k$ and $\sum_j n_j = N$. 
    Define 
    \begin{align*}
        \vec X = \mathrm{diag}(\underbrace{x_1,\dots,x_1}_{n_1},\underbrace{x_2,\dots,x_2}_{n_2},\ldots,\underbrace{x_m,\dots,x_m}_{n_m}), \quad \vec V^* = \begin{bmatrix} \vec V_1 & \cdots & \vec V_m \end{bmatrix}, \quad \vec V \in \mathbb C^{N \times k}.
    \end{align*}
    The measure $\bmu$ is the spectral measure of a matrix in $\JkN$ iff the matrix
    \begin{align*}
        \begin{bmatrix}
            \vec V & \vec X \vec V & \dots & \vec X^d \vec V
        \end{bmatrix} \in \mathbb C^{N\times (d+1)k},
    \end{align*}
    is full rank.
\end{corollary}

\section{Toda Flow on Banded Hermitian Matrices}\label{sec:toda}

In this section, we extend the Toda flow from Section~\ref{sec:todabackground} to the class of banded Hermitian matrices $\JkN$ defined in \eqref{def:JkN}. We show that the fundamental properties of the classical Toda lattice persist in this broader setting, and that the corresponding matrix-valued spectral measure admits a similar evolution to equations \eqref{eq:JacSpectralMeasure} and \eqref{eq:JacSpectralMeasure}.

First, we show that the solution of the Toda flow \eqref{eq:toda} can be expressed 
in terms of a matrix exponential involving the initial condition. For completeness, 
we include a proof following the approach described in \cite{DeiftDubachTomeiTrogdon2025}. We begin by stating a lemma that sets the notation and introduces an important element of the proof.

\begin{lemma}
Every square matrix $\vec A$ admits a unique decomposition of the form
\begin{align*}
    \vec A = \pi_S(\vec A) + \pi_U(\vec A),
\end{align*}
where $\pi_S(\vec A)$ is skew-Hermitian and $\pi_U(\vec A)$ is upper triangular.
\end{lemma}

\begin{proof}
The existence argument follows directly by taking
\begin{align*}
    (\pi_S(\vec A))_{ij} = 
    \begin{cases}
    \vec A_{ij} & \text{if } j > i, \\
    0 & \text{if } i = j, \\
    -\vec A_{ji}^* & \text{if } j < i,
    \end{cases}
    \quad \text{and} \quad
    (\pi_U(\vec A))_{ij} = 
    \begin{cases}
    0 & \text{if } j > i, \\
    \vec A_{ij} & \text{if } i = j, \\
    \vec A_{ij} + \vec A_{ji}^* & \text{if } j < i.
    \end{cases}
\end{align*}
For uniqueness, suppose $\vec A = \vec S_1 + \vec U_1 = \vec S_2 + \vec U_2$, where $\vec S_1, \vec S_2$ are skew-Hermitian and $\vec U_1, \vec U_2$ are upper triangular. It follows that
\begin{align*}
    \vec S_1 - \vec S_2 = \vec U_2 - \vec U_1.
\end{align*}
The left-hand side is skew-Hermitian, while the right-hand side is upper triangular. The only matrix that is both skew-Hermitian and upper triangular is the zero matrix, which shows that the decomposition is unique.
\end{proof}

\begin{theorem}\label{thm:todaSol}
    Let $\vec X_0\in \mathbb C^{N\times N}$ be a Hermitian matrix, and let $\vec X(t)\in\mathbb C^{N\times N}$ denote the solution to the Toda flow \eqref{eq:toda} with initial condition $\vec X_0$. The solution $\vec X(t)$ can be expressed as
    \begin{align}\label{eq:todaSol}
        \vec X(t) = \vec Q^*(t) \vec X_0 \vec Q(t) = \vec R(t) \vec X_0 \vec R^{-1}(t),
    \end{align}
    where $\vec Q(t) \vec R(t)$ is the QR decomposition of $\exp(t\vec X_0)$, with $\vec Q(t)$ unitary and $\vec R(t)$ upper triangular with positive diagonal entries.
\end{theorem}

\begin{proof}
    We first establish the equality $\vec X = \vec Q^* \vec X_0 \vec Q$. Since $\exp(t\vec X_0)$ depends smoothly on $t$, its QR-decomposition is differentiable \cite{dieci_compuation_1997,dieci_smooth_1999}. Differentiating $\exp(t\vec X_0) = \vec Q\vec R$
    gives
    \begin{align}\label{eq:DerivQRToda}
        \partial_t \vec Q \vec R + \vec Q \partial_t \vec R = \vec X_0 \vec Q \vec R.
    \end{align}
    Multiplying on the left by $\vec Q^*$ and on the right by $\vec R^{-1}$, we get
    \begin{align*}
    \vec Q^* \partial_t \vec Q + \partial_t \vec R \vec R^{-1} = \vec Q^* \vec X_0 \vec Q.
    \end{align*}
    Note that $\vec Q^*\vec Q=\vec I_N$, which implies
    \begin{align*}
        \partial_t \vec Q^*\vec Q = \vec Q^*\partial_t \vec Q = 0,
    \end{align*}
    so $\vec Q^*\partial_t \vec Q$ is skew-Hermitian and thus 
    \begin{align*}
        \pi_{S}\left(\vec Q^*\vec X_0\vec Q\right) = \vec Q^*\partial_t \vec Q, \quad \text{and} \quad \pi_{U}\left(\vec Q^*\vec X_0\vec Q\right) = \partial_t \vec R\vec R^{-1}.
    \end{align*}

    Now, define $\tilde{\vec X} = \vec Q^*\vec X_0\vec Q$ and note that $\tilde{\vec X}(0) = \vec X_0$ and
    \begin{align*}
        \partial_t\tilde{\vec X} &= \partial_t \vec  Q^* \vec X_0 \vec Q+\vec Q^*\vec X_0\partial_t \vec Q\\
        &= \left(\partial_t \vec Q^*\vec Q\right)\vec Q^*\vec X_0\vec Q + \vec Q^*\vec X_0\vec Q\left(\vec Q^*\partial_t \vec Q\right)\\
        &= \vec Q^*\vec X_0\vec Q\left(\vec Q^*\partial_t \vec Q\right) -\left(\vec Q^*\partial_t \vec Q\right)\vec Q^*\vec X_0 \vec Q\\
        &= [\tilde{\vec X},\pi_S(\tilde{\vec X})].
    \end{align*} 
    On the other hand, we have $\vec X = (\vec X-\vec B(\vec X)) + \vec B(\vec X)$ which implies $\pi_U(\vec X) = \vec X-\vec B(\vec X)$ and $\pi_S(\vec X) = \vec B(\vec X)$. Substituting into \eqref{eq:toda} gives
    \begin{align*}
        \partial_t \vec X = [\vec X,\pi_S(\vec X)].
    \end{align*}
    Thus, both $\vec X$ and $\tilde{\vec X}$ satisfy the same differential equation with identical initial condition which establishes the first equality in \eqref{eq:todaSol}. The second equality follows from the fact that $\vec X_0$ commutes with $\exp(t\vec X_0) = \vec Q\vec R$. Specifically, we have $\vec X_0 \vec Q \vec R = \vec Q\vec R\vec X_0$ which gives $\vec Q^*\vec X_0\vec Q = \vec R\vec X_0\vec R^{-1},$
    thereby completing the proof.
\end{proof}

\begin{corollary}
    The Toda flow defines an isospectral flow on matrices in $\JkN$. Moreover, it preserves the band size and the structure of $\vec X_0$, that is $\vec X(t)\in \JkN$ for all $t$.
\end{corollary}
\begin{proof}
    The first equality $\vec X = \vec Q^* \vec X_0 \vec Q$ in Theorem~\ref{thm:todaSol} shows that $\vec X$ and $\vec X_0$ are similar, so \eqref{eq:toda} is an isospectral flow. The structure of the solution also guarantees that $\vec X$ remains Hermitian, and the second equality in \eqref{eq:todaSol} implies that $\vec X_{ij} = 0$ whenever $i - j > k$. Together, these properties show that $X(t)$ is banded with bandwidth $k$ for all $t$. 
    
    A closer look at $\vec X = \vec R \vec X_0\vec R^{-1}$ shows that the off-diagonal blocks $\{\vec B_j\}_{j=0}^{n-2}$ of $\vec X$ satisfy $\vec B_j = \vec R_j \vec B_j(0)\vec U_j$ where $\vec R_j$ and $\vec U_j$ are upper triangular matrices with positive diagonal entries. Let $p_i$ denote the pivot column of row $i$ in $\vec B_j(0)$, so that $(\vec B_j(0))_{i,\ell} = 0$ for $\ell < p_i$ and $p_1 < p_2 < \cdots$. For $\ell < p_i$, we have
    \begin{align*}
        (\vec R_j \vec B_j(0))_{i,\ell} = \sum_{k \geq i} (\vec R_j)_{i,k} (\vec B_j(0))_{k,\ell} = 0,
    \end{align*}
    since $(\vec B_j(0))_{k,\ell} = 0$ for all $k \geq i$, as $p_k \geq p_i > \ell$. A similar argument shows that right multiplication by $\vec U_j$ preserves zeros to the left of each pivot. Thus $\vec B_j$ has the same pivot structure as $\vec B_j(0)$, and the pivot entries satisfy
    \begin{align*}
        (\vec B_j)_{i,\,p_i} = (\vec R_j)_{i,i}\,(\vec B_j(0))_{i,\,p_i}\,(\vec U_j)_{p_i,\,p_i} > 0.
    \end{align*}
    We conclude that $\vec X \in \JkN$, completing the proof.
\end{proof}

The next result describes the time evolution of the spectral measure associated with $\vec X(t)$. It shows that the unnormalized weights evolve by an exponential scaling of the initial weights, after which normalization is achieved through the matrix $\vec L(t)$.
\begin{theorem}\label{thm:MeasEvol}
    Suppose that $\vec X(t)$ is a solution to the Toda flow, $\vec X(0) \in \mathcal J_{N,k}$, and denote the spectral measure of $\vec X(t)$ by 
    \begin{align*}
        \bmu_{\vec X(t)} = \sum_{j=1}^m \vec V_j(t) \vec V_j^*(t)\delta_{\lambda_j}, \quad \vec V_j(t)\in \mathbb R^{k\times \ell_j}.
    \end{align*}
    Then
    \begin{align*}
         \vec V_j(t) \vec V_j^*(t) = \vec L^{-1}(t)\left( e^{2\lambda_j t} \vec V_j(0)\vec V_j(0)^*\right) \vec L^{-*}(t),
    \end{align*}
    where $\vec L^{-1}(t) = \vec I_{N\times k}^* \vec R^{-*}(t)\vec I_{N\times k}$
    and $\vec R(t)$ is defined in Theorem~\ref{thm:todaSol}. Equivalently, $\vec L(t)$ is the lower triangular matrix satisfying $\sum_{j} e^{2\lambda_j t} \vec V_j(0)\vec V_j(0)^* = \vec L(t) \vec L(t)^*$.
\end{theorem}
\begin{proof}
    Let $\vec X = \vec U \vec \Lambda \vec U^*$ be the eigendecomposition of $\vec X$, and let $\vec S_j\in \mathbb R^{N\times \ell_j}$ be a column selection matrix that extracts the columns of $\vec U$ corresponding to the eigenvalue $\lambda_j$. In other words, $\vec S_j$ satisfies
    \begin{align*}
        \vec X \vec U \vec S_j = \lambda_j \vec U \vec S_j.
    \end{align*}
    The evolution of the eigenvectors follows directly from equation \eqref{eq:toda}. In particular, 
    \begin{align*}
        \vec U = \vec Q^* \vec U(0) \vec O,
    \end{align*}
    where $\vec O(t)$ is a unitary matrix such that $\vec O_j(t) = \vec S_j^T\vec O(t)\vec S_j$ is also unitary. Since $\vec V_j$ is given by the first $k$ rows of the corresponding eigenvectors, i.e. $\vec V_j = \vec I_{N\times k}^* \vec U \vec S_j$, we have
    \begin{align}\label{eq:BandMeasPfEq1}
        \vec V_j = \vec I_{N\times k}^*\vec Q^* \vec U(0) \vec O \vec S_j.
    \end{align}

    Recall that $\exp(t\vec X_0) = \vec Q \vec R$ where $\vec Q$ is unitary and $\vec R$ is upper triangular, so $\vec Q^* = \vec R^{-*}\exp(t\vec X_0)$ which implies
    \begin{align}\label{eq:BandMeasPfEq2}
        \vec I_{N\times k}^* \vec Q^* = \vec I_{N\times k}^* \vec R^{-*}\exp(t\vec X_0) = \vec L^{-1} \vec I_{N\times k}^* \vec U(0)\exp(t\vec \Lambda)\vec U(0)^*, \quad \text{where} \quad \vec L^{-1} = \vec I_{N\times k}^*\vec R^{-*}\vec I_{N\times k}.
    \end{align}
    Combining both \eqref{eq:BandMeasPfEq1} and \eqref{eq:BandMeasPfEq2} with the fact that $\vec O \vec S_j = \vec S_j \vec O_j$, we find
    \begin{align*}
        \vec V_j = \vec L^{-1} \vec I_{N\times k}^* \vec U(0)\exp(t\vec \Lambda) \vec S_j \vec O_j = \vec L^{-1} \left(e^{\lambda_j t}\vec U_k(0)\vec S_j\right) \vec O_j = \vec L^{-1} \left(e^{\lambda_j t }\vec V_j(0)\right) \vec O_j,
    \end{align*}
    and therefore
    \begin{align*}
        \vec V_j \vec V_j^* =\vec L^{-1} \left( e^{2\lambda_j t} \vec V_j(0)\vec V_j^*(0)\right) \vec L^{-*},
    \end{align*}
    which completes our proof.
\end{proof}
\begin{remark}
    The result in Theorem~\ref{thm:MeasEvol} holds for any Hermitian initial data, however its importance is when $\vec X(0) \in \JkN$, since in that case $\bmu_{\vec X}$ can be used to recover the solution $\vec X$ using the inverse spectral theory developed in Section~\ref{subsec:InvSpectMap}.
\end{remark}
\begin{remark}
    When $k=1$, Theorem~\ref{thm:MeasEvol} reduces to the classical evolution of the spectral measure for Jacobi matrices.  
    In this case, each $\vec V_j(t)$ becomes a scalar $\vec v_{1,j}(t)$, and $\vec L(t)$ simplifies to  
    \begin{align*}
    L(t) = \Bigg(\sum_{i=1}^N e^{2\lambda_i t} v_{1,j}^2(0)\Bigg)^{1/2}.
    \end{align*}
    Substituting into the formula of Theorem~\ref{thm:MeasEvol}, we find the standard expression in \eqref{eq:JacSpectralMeasure}.
\end{remark}

\section*{Acknowledgements}

The authors thank Maxim Yattselev for pointing out several relevant references. This material is based upon work supported by NSF DMS-2306438 (TT). Any opinions, findings, and conclusions or recommendations expressed in this material are those of the author and do not necessarily reflect the views of the National Science Foundation.

\appendix

\section{Block Lanczos and Householder}\label{appendix:Algs}

\begin{algorithm}[H]
	\caption{Block Lanczos Algorithm}
    \label{a:BlockLanczos}
    \begin{algorithmic}[1]
        \Statex \textbf{Input:} Hermitian matrix $\vec A\in \mathbb{C}^{N\times N}$, initial block $\vec V\in \mathbb C^{N\times k}$, and number of iteration $n$.
        \State  $\vec V_1,\vec B_{-1} = \mathrm{qr}(\vec V)$
        \For{$j=0,1,\dots,n-1$} 
            \State $\vec Z = \vec A \vec V_{j+1} - \vec V_{j}\vec B_{j-1}^*$ \RightComment{if $j=0$, $\vec Z = \vec A\vec V_1$}
            \State $\vec A_j = \vec V_{j+1}^* \vec Z$
            \State $\vec Z = \vec Z-\vec V_{j+1}\vec A_j$
            \State $\vec Z = \vec Z-\sum_{i=0}^{j+1}\vec V_i\vec V_i^* \vec Z$ \RightComment{reorthogonalize (optional)}
            \State $\vec V_{j+2},\vec B_j = \mathrm{qr}(\vec Z)$\footnotemark
        \EndFor
        \State \Return{$\{\vec A_j\},\{\vec B_j\},\{\vec V_j\}$}
    \end{algorithmic}
\end{algorithm}
\footnotetext{%
If $\vec Z \in \mathbb{C}^{k\times k}$ has rank $r$, then
$\mathrm{qr}(\vec Z)$ returns $\vec Q\in\mathbb{C}^{k\times r}$ with orthonormal columns and $\vec R\in\mathbb{C}^{r\times k}$ upper triangular with positive diagonal entries.
}

\begin{algorithm}[H]
	\caption{Householder reduction to block tridiagonal form}
    \label{a:BlockHouseholder}
    \begin{algorithmic}[1]
        \Statex \textbf{Input:} Hermitian matrix $\vec A\in \mathbb{C}^{N\times N}$, and block size $k$.
        \For{$j=1,\dots,N-k-1$} 
            \State $\vec v = \vec A_{k+j:N,j}$\footnotemark
            \State $\vec u = |v_1| \|\vec v\|\vec e_1 /v_1 + \vec v$, where $v_1 = \vec e_1^*\vec v$
            \State $\vec w = \vec u/\|\vec u\|$
            \State $\vec A_{k+j,j:N} = \vec A_{k+j,j:N}-2\vec w(\vec w^*\vec A_{k+j,j:N)}$
            \State $\vec A_{j:N,k+j:N} = \vec A_{j:N,k+j:N} - 2(\vec A_{j:N,k+j:N} \vec w)\vec w^*$
        \EndFor
        \State \Return $\vec A_{\,|i-j|\le k}$
    \end{algorithmic}
\end{algorithm}
\footnotetext{%
The notation $\vec A_{p:q,r:s}$ denotes the submatrix of $\vec A$ consisting of rows $p,\dots,q$ and columns $r,\dots,s$.
}

\bibliographystyle{abbrv}
\bibliography{library}  

\end{document}